\documentclass[11pt]{amsart}
\usepackage{amsfonts,amssymb, amsthm,amsmath}
\usepackage{indentfirst}
\usepackage{amssymb,mathrsfs,graphicx,subfigure, enumerate}
\usepackage{amsmath,amsfonts,amscd,amsthm,bbm}
\usepackage{extpfeil}
\usepackage{graphicx,colortbl}
\usepackage{graphicx}
\usepackage{subfigure}
\usepackage {stfloats}
\usepackage{pstricks}
\usepackage{float}
\usepackage{epsfig}
\usepackage{caption}
\usepackage[pagewise]{lineno}
\usepackage{bm}
\graphicspath{{Figure/}}
\usepackage[english]{babel}
\usepackage{tikz}
\usepackage[title]{appendix}
\usepackage[T1]{fontenc}
\usepackage[colorlinks=true]{hyperref}
\hypersetup{urlcolor=red, citecolor=blue}
\hypersetup{linkcolor=black}

\allowdisplaybreaks[4]
\usepackage[pagewise]{lineno}

\title[Consensus and Flocking with Memory]{Exponential Consensus and Flocking in Multi-Agent Systems with Infinite Fading Memory}

\author[C. Pignotti]{Cristina Pignotti}
\address[Cristina Pignotti]{\newline Dipartimento di Ingegneria e Scienze dell'Informazione e Matematica, Universit\`{a} dell'Aquila, 67100 L'Aquila, Italy}
\email{cristina.pignotti@univaq.it}

\author[Y.-Q. Wang]{Yu-Qing Wang}
\address[Yu-Qing Wang]{\newline School of Mathematical Sciences, Dalian University of Technology, Dalian 116024, China \newline \& \newline Dipartimento di Ingegneria e Scienze dell'Informazione e Matematica, Universit\`{a} dell'Aquila, 67100 L'Aquila, Italy}
\email{yqwang202309@163.com}

\newtheorem{theorem}{Theorem}[section]
\newtheorem{lemma}{Lemma}[section]

\newtheorem{remark}{Remark}[section]
\newtheorem{definition}{Definition}[section]
\newtheorem{assumption}{Assumption}

\numberwithin{equation}{section}

\newcommand{\x}{\bm x}

\def\charf {\mbox{{\text 1}\kern-.24em {\text l}}}

\begin{document}
\date{\today}

\subjclass[2020]{34D05, 34K25, 45J05, 45M05, 92D25.}
\keywords{Alignment models, Fading memory, Consensus, Flocking.}

\begin{abstract}
In this paper, we study the emergent collective dynamics of multi-agent systems driven by infinite distributed fading memory of Volterra type. We establish a unified theoretical framework covering both first-order opinion consensus dynamics and second-order velocity alignment flocking kinematics. By introducing Dafermos past-history transformations, the governing integro-differential systems are reformulated into dynamical systems on extended product Hilbert spaces. For first-order dynamics, we prove that fading memory inherently provides a hidden dissipative mechanism, guaranteeing unconditional global exponential consensus with or without instantaneous communication forces. {For second-order dynamics, by developing a unified energy-multiplier framework that seamlessly couples instantaneous communication with memory dissipation, we establish unconditional global exponential flocking for both the pure fading memory system and the mixed interaction system, completely independent of the spatial dispersion or the communication tail.}
\end{abstract}

\maketitle \centerline{\date}

\tableofcontents

\section{Introduction}

In recent years, multi-agent systems have attracted significant attention across several scientific disciplines, including robotics, economics, control theory, mathematical biology, and social sciences \cite{Bullo2009, Camazine2001, Jackson2008, Jadbabaie2003, OlfatiSaber2007}. A fundamental aspect of multi-agent dynamics is understanding how simple self-organizing interaction rules lead to the emergence of globally coordinated collective behaviors, such as opinion consensus in human societies and flocking or schooling in biological swarms \cite{Castellano2009, Vicsek1995}. To describe continuous opinion formation and decision-making processes in social networks, Hegselmann and Krause \cite{Hegselmann2002} proposed the celebrated Hegselmann--Krause (HK) bounded confidence model:
\begin{equation}\label{eq:hk_classic}
\dot{\boldsymbol{x}}_i(t) = \frac{1}{N} \sum_{k=1}^{N} a_{ik}(t)\bigl(\boldsymbol{x}_k(t) - \boldsymbol{x}_i(t)\bigr), \qquad i \in [N] := \{1, \ldots, N\},
\end{equation}
where $\boldsymbol{x}_i(t) \in \mathbb{R}^n$ represents the continuous opinion state of the $i$-th agent at time $t$, and $a_{ik}(t) = \mathbf{1}_{\{\|\boldsymbol{x}_i(t) - \boldsymbol{x}_k(t)\| \le \epsilon\}}$ denotes the state-dependent interaction weight bounded by a confidence threshold $\epsilon > 0$. In this framework, agents adjust their opinions selectively by averaging over neighbors whose opinions lie within their confidence range \cite{Canuto2012, Jabin2014}. Several  recent extensions consider always positive interactions and more general interaction coefficients \cite{Choi2021, Haskovec2021}.

On the other hand, to model second-order swarming phenomena and velocity alignment in mobile physical or biological groups, Cucker and Smale \cite{Cucker2007} established the canonical flocking paradigm, which was subsequently generalized to second-order velocity alignment dynamics governed by abstract communication functions:
\begin{equation}\label{eq:cs_classic}
\begin{cases}
\dot{\boldsymbol{x}}_i(t) = \boldsymbol{v}_i(t), \\[4pt]
\dot{\boldsymbol{v}}_i(t) = \displaystyle\frac{1}{N} \sum_{k=1}^{N} \psi(\|\boldsymbol{x}_i(t) - \boldsymbol{x}_k(t)\|)\bigl(\boldsymbol{v}_k(t) - \boldsymbol{v}_i(t)\bigr), \qquad i \in [N],
\end{cases}
\end{equation}
where $\boldsymbol{x}_i(t) \in \mathbb{R}^n$ and $\boldsymbol{v}_i(t) \in \mathbb{R}^n$ represent the spatial position and velocity vectors of agent $i$, respectively. Here, the communication function $\psi : \mathbb{R}_+ \to \mathbb{R}_+$ is an abstract, non-increasing, strictly positive function quantifying the distance-dependent interaction strength between agents \cite{Ha2009}. Both the HK opinion framework \eqref{eq:hk_classic} and second-order velocity alignment kinematics \eqref{eq:cs_classic} have been extensively investigated regarding asymptotic consensus and cluster formation \cite{Choi2021, Jabin2014}, velocity flocking across fixed and switching topologies \cite{Dong, Fanqin}, as well as their kinetic descriptions and mean-field limits \cite{Carrillo2010, Ha2009, Ha2}.

Despite the analytical success of classical HK and CS models, standard formulations rely heavily on local-in-time interactions, assuming that agent state adjustments depend strictly on instantaneous state evaluations. In practical, social, biological, and technological networks, information transmission and cognitive processing naturally involve non-negligible time lags and cognitive inertia \cite{Chiara, Choi2021, Continelli, Continelli2023, Du2025, Haskovec2021, Pignotti2018, Cartabia}. To account for temporal delay effects, several delayed multi-agent models have been proposed in the literature:
\begin{equation*}\label{eq:hk_delay}
\dot{\boldsymbol{x}}_i(t) = \sum_{k=1}^{N} a_{ik}(t)\bigl(\boldsymbol{x}_k(t - \tau(t)) - \boldsymbol{x}_i(t)\bigr),
\end{equation*}
where $\tau(t) \ge 0$ represents a pointwise or time-varying communication delay \cite{Choi2021, Continelli2023}.

However, pointwise time delays consider information from only a single past instant $t - \tau(t)$. In practical multi-agent networks, decision-making and collective interactions naturally rely on the continuous accumulation of historical experience \cite{Jiang2025, Liu2023}. In social dynamics, the history of mutual agreement or disagreement fundamentally shapes how individuals evaluate information, giving rise to an emergent collective memory \cite{Boschi2021}. Recent studies also demonstrate that individual memory capacities quantitatively govern opinion propagation and consensus formation in complex networks \cite{Liu2026}. Beyond opinion dynamics, historical memory has recently become an essential mechanism in artificial intelligence, where agents utilize workflow memory abstracted from past action trajectories to guide complex task execution \cite{wang2025agent}.

Mathematically, hereditary memory effects spanning an infinite history are classically modeled using Volterra-type integro-differential equations in viscoelasticity, heat conduction, and continuum physics (see, e.g., \cite{Chepyzhov2006, Conti2020, Dafermos1970, Giorgi2001}). In multi-agent synchronization, memory effects have recently been incorporated into phase oscillator ensembles such as the Kuramoto model \cite{Cho2025}. However, existing synchronization frameworks restrict the memory integral to the finite interval $[0, t].$ We refer to \cite{Hask_memory} for recent results about clustering in spontaneous particle aggregation with finite memory. {For linear Volterra equations with a single scalar memory kernel, the asymptotic stability of abstract operator equations has been investigated using semigroup and past-history methods (see, e.g., \cite{Chepyzhov2006, Conti2020}). However, in multi-agent networks where memory kernels are pair-dependent and communication may involve nonlinear state-dependent interactions, the collective emergent dynamics remain less understood.}

Inspired by Volterra's hereditary principles \cite{Volterra1930} and fading memory theories \cite{Conti2020, Dafermos1970, Giorgi2001}, this paper establishes a unified theoretical framework for multi-agent opinion consensus and second-order velocity alignment flocking with infinite fading memory. Specifically, we investigate two fundamental classes of systems: first-order opinion dynamics and second-order velocity alignment dynamics. For first-order opinion dynamics, historical relative states are accumulated through a symmetric memory kernel $g_{ik}(s)$. We formulate the following two first-order models:
\vspace{0.2cm}

\noindent\textbf{Model I (First-Order Pure Memory):}
\begin{equation}\label{eq:model1}
\left\{
\begin{aligned}
&\dot{\boldsymbol{x}}_i(t) = \sum_{k=1}^{N} \int_{0}^{\infty} g_{ik}(s)\bigl(\boldsymbol{x}_k(t-s)-\boldsymbol{x}_i(t-s)\bigr)\,ds,~~ i \in [N],\\
&\boldsymbol{x}_i(t)=\boldsymbol{x}^{in}_i(t), \quad t\in(-\infty,0].
\end{aligned}
\right.
\end{equation}
Note that the memory integral covers the entire infinite past. Therefore, the global well-posedness of the integro-differential equations inherently necessitates the initial history to be well-defined over $(-\infty, 0]$. Here, the initial data $\x_i^{in}(t),~\forall i\in[N]$ are assumed to be bounded and continuous functions on $(-\infty, 0]$.

While Model I captures dynamics driven purely by historical accumulation, practical multi-agent interactions often involve a combination of cognitive memory and immediate environmental stimuli. To account for the coexistence of both mechanisms, we incorporate instantaneous interactions, leading to:
\vspace{0.2cm}

\noindent\textbf{Model II (First-Order Instantaneous Interaction + Memory):}
{\small
\begin{equation}\label{eq:model2}
\left\{
\begin{aligned}
&\dot{\boldsymbol{x}}_i(t) = \sum_{k=1}^{N} \psi(\|\boldsymbol{x}_i(t)-\boldsymbol{x}_k(t)\|) \bigl(\boldsymbol{x}_k(t)-\boldsymbol{x}_i(t)\bigr) + \sum_{k=1}^{N}\int_{0}^{\infty} g_{ik}(s) \bigl(\boldsymbol{x}_k(t-s)-\boldsymbol{x}_i(t-s)\bigr)\,ds,\\
&\boldsymbol{x}_i(t)=\boldsymbol{x}^{in}_i(t),~~t\in(-\infty,0],
\end{aligned}
\right.
\end{equation}
}
where $\psi : \mathbb{R}_+ \to \mathbb{R}_+$ denotes a continuous, non-increasing, strictly positive communication function. The initial history $\x_i^{in}(t),~\forall i\in[N]$ satisfies the same boundedness and continuity conditions as in Model I.

Beyond first-order opinion dynamics, we extend our framework to second-order velocity alignment kinematics ($\dot{\boldsymbol{x}}_i(t) = \boldsymbol{v}_i(t)$), where the acceleration is driven by relative velocity interactions. Similar to the first-order case, we first investigate a purely memory-driven velocity alignment model:
\vspace{0.2cm}

\noindent\textbf{Model III (Second-Order Pure Memory):}
\begin{equation}\label{eq:model3}
\left\{
\begin{aligned}
&\dot{\boldsymbol{x}}_i(t) = \boldsymbol{v}_i(t), \\
&\dot{\boldsymbol{v}}_i(t) = \displaystyle\sum_{k=1}^{N} \int_{0}^{\infty} g_{ik}(s)\bigl(\boldsymbol{v}_k(t-s)-\boldsymbol{v}_i(t-s)\bigr)\,ds,~~i \in [N].\\
\end{aligned}
\right.
\end{equation}
Because the infinite memory integral exclusively acts on the velocity components, a bounded and continuous initial history is only required for the velocities, while the spatial positions necessitate only point initial values. The initial conditions are thus specified as:
\begin{equation*}\label{initial-eq-model3}
\left\{
\begin{aligned}
&\boldsymbol{x}_{i}(0) = \boldsymbol{x}_{i}^{0},\\
&\boldsymbol{v}_i(t)=\boldsymbol{v}^{in}_i(t), \quad t\in(-\infty,0].
\end{aligned}
\right.
\end{equation*}

Finally, incorporating classic state-dependent instantaneous communication yields the full mixed-interaction flocking model:
\vspace{0.2cm}

\noindent\textbf{Model IV (Second-Order Instantaneous Interaction + Memory):}
{\small
\begin{equation}\label{eq:model4}
\begin{cases}
\dot{\boldsymbol{x}}_i(t) = \boldsymbol{v}_i(t), \\[4pt]
\dot{\boldsymbol{v}}_i(t) = \displaystyle\sum_{k=1}^{N} \psi(\|\boldsymbol{x}_i(t)-\boldsymbol{x}_k(t)\|) \bigl(\boldsymbol{v}_k(t)-\boldsymbol{v}_i(t)\bigr) + \sum_{k=1}^{N} \int_{0}^{\infty} g_{ik}(s)\bigl(\boldsymbol{v}_k(t-s)-\boldsymbol{v}_i(t-s)\bigr)\,ds.
\end{cases}
\end{equation}
}
The initial conditions for Model IV identically mirror those of Model III, requiring $\boldsymbol{x}_{i}(0) = \boldsymbol{x}_{i}^{0}$ and a bounded continuous history $\boldsymbol{v}_i(t)=\boldsymbol{v}^{in}_i(t)$ for $t\in(-\infty,0]$.

A central theoretical goal of this paper is to find a general framework ensuring   multi-agent systems driven by Volterra fading memory exhibit global exponential state consensus or velocity flocking. {When the memory kernels are homogeneous across the entire network, namely $g_{ik}(s) = a_{ik} g(s)$ with a scalar kernel $g(s)$, Model~I reduces on the zero-mean subspace to an abstract linear Volterra integro-differential equation governed by a weighted graph Laplacian, which falls within the setting studied in \cite{Chepyzhov2006, Conti2020}. In the present work, we address several aspects that extend beyond this scalar linear framework:
\begin{itemize}
 \item [(i)] Fully pair-dependent memory kernels $g_{ik}(s)$, where the decay rates and weights vary across different pairs of agents and cannot be factored into a single scalar kernel;
 \item [(ii)] The presence of nonlinear state-dependent communication weights $\psi(\|\boldsymbol{x}_i(t) - \boldsymbol{x}_k(t)\|)$ in Models~II and IV;
 \item [(iii)]  Second-order velocity alignment dynamics where spatial expansion and velocity dissipation are coupled.
\end{itemize}
}
{Regarding second-order flocking kinematics, we note that distributed reaction delays spanning the infinite past have been investigated by Haskovec and Markou \cite{Haskovec2020}. It is worth emphasizing the structural and analytical distinctions between their framework and our Model~IV:
in \cite{Haskovec2020}, the communication weights depend on past spatial positions, i.e., $\psi(\|\boldsymbol{x}_i(t-s) - \boldsymbol{x}_k(t-s)\|)$, which necessitates uniform bounds and restrictive geometric assumptions on the historical spatial trajectories over $(-\infty, 0]$ to prevent weight degeneration. In contrast, our Model~IV adopts an additive coupling that separates instantaneous spatial communication from hereditary velocity memory. As a consequence, spatial positions require only point initial data $\boldsymbol{x}_i(0)$ rather than a whole historical trajectory, and our energy-multiplier approach establishes unconditional global exponential flocking without imposing geometric restrictions on the initial configuration.}

The mathematical analysis of systems \eqref{eq:model1}--\eqref{eq:model4} presents major analytical challenges. Because the interactions involve memory integrals extending over $(0, \infty),$ standard Gr\"onwall-type estimates or energy methods on finite-dimensional phase spaces \cite{Ha2009, Cartabia} cannot be directly applied, as the state derivative depends explicitly on the entire historical trajectory. Furthermore, in pure memory dynamics (Models I and III) where instantaneous coupling is absent ($\psi \equiv 0$), standard kinetic energy derivatives lack state or velocity dissipation terms, making it difficult to establish asymptotic convergence. To overcome these technical obstacles, we establish a novel analytical approach based on the following key ingredients:
\begin{enumerate}
    \item [(i)] \textbf{Dafermos History Variable Transformation:} Inspired by the methodology introduced by Dafermos in viscoelasticity \cite{Dafermos1970}, we define the position and velocity past-history variables $$\boldsymbol{\eta}_i^t(s) := \int_{t-s}^t \boldsymbol{x}_i(r)\,dr, \quad \boldsymbol{\xi}_i^t(s) := \int_{t-s}^t \boldsymbol{v}_i(r)\,dr.$$
     By performing integration by parts on the Volterra memory integrals, the non-autonomous integro-differential models are transformed into autonomous dynamical systems on extended Hilbert spaces governed by linear transport operators $\partial_t \boldsymbol{\eta}_i^t(s) = \boldsymbol{x}_i(t) - \partial_s \boldsymbol{\eta}_i^t(s)$ and $\partial_t \boldsymbol{\xi}_i^t(s) = \boldsymbol{v}_i(t) - \partial_s \boldsymbol{\xi}_i^t(s)$.
    \item [(ii)] \textbf{Auxiliary Multipliers and Hidden Dissipation:} To address the lack of instantaneous state dissipation in Models I and III, we construct auxiliary multiplier functionals ($\Phi(t)$ for first-order and $\Phi_v(t)$ for second-order). By differentiating these multipliers along trajectories and utilizing Cauchy--Schwarz and Young inequalities, we extract the "hidden dissipation" generated by the kernel derivative $\mu'_{ik}(s) \le -\delta \mu_{ik}(s)$, where $\mu_{ik}(s) = -g'_{ik}(s)$. Combining these multipliers with the standard Dafermos energy, we construct strictly coercive Lyapunov functionals $\mathcal{E}(t)$ and $\mathcal{E}_v(t)$ that prove unconditional global exponential convergence.
    \item [(iii)] {\textbf{Unified Multiplier Framework for Mixed Flocking Dynamics:} For Model IV, where instantaneous communication $\psi(\|\boldsymbol{x}_i(t)-\boldsymbol{x}_k(t)\|)$ is coupled with velocity fading memory, spatial dispersion could potentially weaken instantaneous alignment forces. Instead of relying on traditional position-integral functionals and restrictive initial capacity thresholds, we establish a unified velocity multiplier framework directly on the coupled system. Since the instantaneous communication Laplacian is positive semi-definite and bounded by $N\psi(0)I$, the resulting cross-terms can be absorbed by the instantaneous dissipation and the hidden dissipation generated by the kernel derivative $\mu'_{ik}(s) \le -\delta \mu_{ik}(s)$. This enables us to establish unconditional global exponential velocity alignment, from which uniform spatial boundedness follows unconditionally as a natural consequence.}
\end{enumerate}

The rest of this paper is organized as follows. Section \ref{sec:pre} presents preliminaries, fundamental assumptions on memory kernels and communication functions, formal mathematical definitions of state consensus and velocity flocking, and detailed proofs of the Dafermos history transformations. Section \ref{sec:main_result} summarizes all main theoretical results regarding well-posedness, exponential consensus, and flocking. In Section \ref{sec:proof_wp}, we provide the complete proof of global well-posedness for both first-order and second-order models. Section \ref{sec:proof_1st} contains the detailed proofs of unconditional global exponential consensus for Models I and II. In Section \ref{sec:proof_2nd}, we present the proofs of conditional and unconditional exponential flocking for Models IV and III. Finally, we draw conclusions in Section \ref{sec:conclusion}.

\section{Preliminaries}\label{sec:pre}

\noindent\textbf{Notation.} In this paper, for column vectors in $\mathbb{R}^n$, we adopt the standard Euclidean norm denoted by $\|\cdot\|$. We set $[N] := \{1, \ldots, N\}$. The ensemble position vector and velocity vector are denoted by $\boldsymbol{x} := (\boldsymbol{x}_1, \ldots, \boldsymbol{x}_N) \in (\mathbb{R}^n)^N$ and $\boldsymbol{v} := (\boldsymbol{v}_1, \ldots, \boldsymbol{v}_N) \in (\mathbb{R}^n)^N$, respectively. We define the maximum spatial diameter and maximum velocity diameter as
\begin{equation*}
	X_{\max}(t) := \max_{i,k\in[N]} \|\boldsymbol{x}_i(t) - \boldsymbol{x}_k(t)\|, \qquad V_{\max}(t) := \max_{i,k\in[N]} \|\boldsymbol{v}_i(t) - \boldsymbol{v}_k(t)\|.
\end{equation*}
Moreover, we denote $\mathbb{R}_+:=[0, +\infty).$ Finally, for all $t\in(-\infty,0],$ we denote the initial data by $\x^{in}(t):=\left(\x^{in}_1(t),\x^{in}_2(t),...,\x^{in}_N(t)\right),~\boldsymbol{v}^{in}(t):=\left(\boldsymbol{v}^{in}_1(t),\boldsymbol{v}^{in}_2(t),...,\boldsymbol{v}^{in}_N(t)\right).$ Especially, for $t=0,$ we further denote
$$
\x^{0}:=\left(\x^{0}_1,\x^{0}_2,...,\x^{0}_N\right):=\left(\x^{in}_1(0),\x^{in}_2(0),...,\x^{in}_N(0)\right).
$$

Now, we present the fundamental standing assumptions on the communication weight functions and memory kernels, formulate the precise mathematical definitions of consensus and flocking, and establish the Dafermos past-history transformations.

\begin{assumption}[Communication Weights]\label{assum:psi}
{\rm
The state-dependent communication function $\psi : \mathbb{R}_+ \to \mathbb{R}_+$ is globally Lipschitz continuous, non-increasing, and strictly positive.
}
\end{assumption}

\begin{assumption}[Fading Memory Kernel]\label{assum:g}
{\rm
For each pair $i,k \in [N]$, the memory kernel $g_{ik} : \mathbb{R}_+ \to \mathbb{R}_+$ satisfies:
\begin{enumerate}
  \item [\rm (i)] \textbf{Symmetry and Non-negativity, and Nondegeneracy:} For all $s \ge 0,$
  $$
  g_{ik}(s) = g_{ki}(s) \ge 0.
  $$
{Moreover, for all off-diagonal interactions ($i \ne k$), the memory kernel is strictly positive at $s = 0$, i.e.,
\[
g_{ik}(0) > 0, \quad \forall i \ne k \in [N],
\]
ensuring an all-to-all positive memory coupling across distinct agents.}
  \item [\rm (ii)] \textbf{Regularity and Decay:} $g_{ik}$ is twice continuously differentiable and non-increasing, with $g_{ik}(s) \rightarrow  0$ as $s\rightarrow \infty.$ Moreover, defining for all $s\ge0,$
  \begin{equation*}\label{eq:mu_def}
  \mu_{ik}(s) := -g_{ik}'(s) \ge 0,
  \end{equation*}
  there exists a uniform constant $\delta > 0$ such that for all $s\ge0,$
  \begin{equation}\label{eq:g_decay}
      \mu_{ik}'(s) + \delta\, \mu_{ik}(s) \le 0.
  \end{equation}
\end{enumerate}
}
\end{assumption}

\begin{remark}\label{rem:kernel_implication}
{\rm
By the Fundamental Theorem of Calculus and the boundary condition $g_{ik}(s) \rightarrow  0$ as $s\rightarrow \infty$, the memory kernel $g_{ik}$ is recovered as the tail integral of $\mu_{ik}$:
\begin{equation*}\label{eq:g_as_tail_of_mu}
g_{ik}(s) = \int_s^\infty \big(-g_{ik}'(\tau)\big)\,d\tau = \int_s^\infty \mu_{ik}(\tau)\,d\tau, \qquad \forall s \ge 0.
\end{equation*}
Integrating \eqref{eq:g_decay} over $[s,\infty)$ yields $-\mu_{ik}(s) + \delta\, g_{ik}(s) \le 0$, which implies the domination property
\begin{equation*} \label{eq:mu_domination}
g_{ik}(s) \le \frac{1}{\delta}\, \mu_{ik}(s), \qquad \forall s\ge0.
\end{equation*}
Gr\"onwall's inequality applied to \eqref{eq:g_decay} guarantees $\mu_{ik}(s)\le \mu_{ik}(0)e^{-\delta s}$, and hence $g_{ik}(s)\le \frac{\mu_{ik}(0)}{\delta}e^{-\delta s}$. We define the maximum total mass parameter of the original kernels by
\begin{equation*}
\kappa_{\max} := \max_{i,k\in[N]} \int_0^\infty g_{ik}(s)\,ds < \infty,
\end{equation*}
and the mass parameters of the auxiliary tail density $\mu_{ik}$ by
$$
\lambda_{ik} := \int_0^\infty \mu_{ik}(s)\,ds = g_{ik}(0) \ge 0, \quad \forall i, k \in [N].
$$
$$
\lambda_{\max} := \max_{i,k\in[N]}\lambda_{ik}, \quad \mu^0_{\max} := \max_{i,k\in[N]} \mu_{ik}(0) < \infty.
$$
{By virtue of the nondegeneracy condition in Assumption \ref{assum:g}(i), the off-diagonal initial masses are strictly positive ($g_{ik}(0) > 0$ for $i \ne k$). We define the minimum positive mass parameter by}
\begin{equation}\label{2.2}
{\lambda_{\min} := \min_{i \ne k} \lambda_{ik} = \min_{i \ne k} g_{ik}(0) > 0.}
\end{equation}
}
\end{remark}

{\begin{remark}[Conservation Laws and Galilean Invariance]\label{rem:conservation_laws}
{\rm
Owing to the anti-symmetry of the relative states $\boldsymbol{x}_k - \boldsymbol{x}_i$, $\boldsymbol{v}_k - \boldsymbol{v}_i$, and the symmetry of communication functions and memory kernels ($g_{ik} = g_{ki}$), summing the governing equations over $i \in [N]$ immediately yields:
\[
\frac{d}{dt}\sum_{i=1}^N \boldsymbol{x}_i(t) = \mathbf{0} \quad \text{(for Models I and II)}, \qquad \frac{d}{dt}\sum_{i=1}^N \boldsymbol{v}_i(t) = \mathbf{0} \quad \text{(for Models III and IV)}.
\]
Consequently, the mean position $\boldsymbol{x}_c := \frac{1}{N}\sum\limits_{i=1}^N \boldsymbol{x}_i(t)\equiv \frac{1}{N}\sum\limits_{i=1}^N \boldsymbol{x}_i(0)$ and the mean velocity $\boldsymbol{v}_c := \frac{1}{N}\sum\limits_{i=1}^N \boldsymbol{v}_i(t) \equiv \frac{1}{N}\sum\limits_{i=1}^N \boldsymbol{v}_i(0)$ are conserved for all $t \ge 0$ for first-order models (Models I--II) and second-order models (Models III--IV), respectively. By translation invariance ($\boldsymbol{x}_i \mapsto \boldsymbol{x}_i - \boldsymbol{x}_c$) for first-order models and Galilean invariance ($\boldsymbol{v}_i \mapsto \boldsymbol{v}_i - \boldsymbol{v}_c$) for second-order models, we can assume without loss of generality that the ensemble is centered:
\[
\sum_{i=1}^N \boldsymbol{x}_i(t) \equiv \mathbf{0} \quad \text{or} \quad \sum_{i=1}^N \boldsymbol{v}_i(t) \equiv \mathbf{0}, \qquad \forall t \ge 0.
\]
In this zero-mean frame, the classical algebraic identity $\sum\limits_{i,k=1}^N \|\boldsymbol{u}_i - \boldsymbol{u}_k\|^2 = 2N\sum\limits_{i=1}^N \|\boldsymbol{u}_i\|^2$ holds for $\boldsymbol{u} \in \{\boldsymbol{x}, \boldsymbol{v}\}$.
}
\end{remark}
}

Next, we introduce the formal mathematical definitions of asymptotic consensus and flocking for the multi-agent systems under study.

\begin{definition}\label{def:consensus}
{\rm
A solution $\boldsymbol{x}(t) = (\boldsymbol{x}_1(t), \ldots, \boldsymbol{x}_N(t))$ to the first-order system \eqref{eq:model1} or \eqref{eq:model2} is said to achieve \textbf{asymptotic  consensus} if the maximum spatial diameter converges to zero asymptotically:
\begin{equation*}
\lim_{t \to \infty} X_{\max}(t) = \lim_{t \to \infty} \max_{i,k\in[N]} \|\boldsymbol{x}_i(t) - \boldsymbol{x}_k(t)\| = 0.
\end{equation*}
If there exist positive constants $C_0 > 0$ and $\gamma > 0$ such that for all $t \ge 0,$
$$
X_{\max}(t) \le C_0 e^{-\gamma t},
$$
then the system achieves \textbf{exponential consensus}.
}
\end{definition}

\begin{definition}\label{def:flocking}
{\rm
A solution $(\boldsymbol{x}(t), \boldsymbol{v}(t)) = (\boldsymbol{x}_1(t), \ldots, \boldsymbol{x}_N(t), \boldsymbol{v}_1(t), \ldots, \boldsymbol{v}_N(t))$ to the second-order system \eqref{eq:model3} or \eqref{eq:model4} is said to achieve \textbf{asymptotic flocking} if the following two conditions are satisfied simultaneously:
\begin{enumerate}
    \item [\rm (i)] \textbf{Velocity Alignment:} The maximum velocity diameter decays to zero asymptotically, i.e.,
    \begin{equation*}
    	\lim_{t \to \infty}
    V_{\max}(t) = 	\lim_{t \to \infty}\max_{i,k\in[N]} \|\boldsymbol{v}_i(t) - \boldsymbol{v}_k(t)\| =0.
    \end{equation*}
    \item [\rm (ii)] \textbf{Uniform Spatial Boundedness:} The maximum spatial diameter remains uniformly bounded for all times:
    \begin{equation*}
    \sup_{t \ge 0} X_{\max}(t) = \sup_{t \ge 0} \max_{i,k\in[N]} \|\boldsymbol{x}_i(t) - \boldsymbol{x}_k(t)\| < \infty.
    \end{equation*}
\end{enumerate}
}
\end{definition}
If the maximum velocity diameter decays to zero exponentially fast,
then we say the system achieves \textbf{exponential flocking}.

Finally, we can rewrite  the integro-differential equations via the Dafermos past-history variable reformulations (see \cite{Dafermos1970}).

\begin{lemma}[First-Order History Transformation]\label{lem:uncoupled_transform_1st}
For $i \in [N]$ and $t \ge 0$, define the position past-history variable:
\begin{equation} \label{eq:eta_uncoupled}
\boldsymbol{\eta}_i^t(s) := \int_{t-s}^{t} \boldsymbol{x}_i(r)\,dr, \qquad s \ge 0.
\end{equation}
Then $\boldsymbol{\eta}_i^t(0) = \mathbf{0}$, $\partial_s \boldsymbol{\eta}_i^t(s) = \boldsymbol{x}_i(t-s)$, and the core transport identity holds strictly:
\begin{equation}\label{eq:transport_uncoupled}
\partial_t \boldsymbol{\eta}_i^t(s) = \boldsymbol{x}_i(t) - \boldsymbol{x}_i(t-s) = \boldsymbol{x}_i(t) - \partial_s \boldsymbol{\eta}_i^t(s).
\end{equation}
Furthermore, the position memory integral transforms into:
\begin{equation*}\label{eq:ibp_uncoupled}
\int_{0}^{\infty} g_{ik}(s)\bigl(\boldsymbol{x}_k(t-s)-\boldsymbol{x}_i(t-s)\bigr)\,ds = \int_0^\infty \mu_{ik}(s) \bigl(\boldsymbol{\eta}_k^t(s) - \boldsymbol{\eta}_i^t(s)\bigr)\,ds.
\end{equation*}
\end{lemma}

\begin{proof}
Identity \eqref{eq:transport_uncoupled} follows directly by differentiating \eqref{eq:eta_uncoupled} with respect to $t$. Applying integration by parts in $s,$ using $\boldsymbol{\eta}_i^t(0)=\mathbf{0}$ and $g_{ik}(s) \rightarrow  0$ as $s\rightarrow \infty$, together with the exponential decay of $g_{ik}(s)$, we have:
{\small
\begin{align*}
\int_{0}^{\infty} g_{ik}(s)\bigl(\boldsymbol{x}_k(t-s)-\boldsymbol{x}_i(t-s)\bigr)\,ds &= \int_0^\infty g_{ik}(s)\, \partial_s \bigl(\boldsymbol{\eta}_k^t(s) - \boldsymbol{\eta}_i^t(s)\bigr)\, ds \\
&= \Big[g_{ik}(s)\big(\boldsymbol{\eta}_k^t(s)-\boldsymbol{\eta}_i^t(s)\big)\Big]_0^\infty - \int_0^\infty g_{ik}'(s)\big(\boldsymbol{\eta}_k^t(s) - \boldsymbol{\eta}_i^t(s)\big)\,ds \\
&= \mathbf{0} + \int_0^\infty \big(-g_{ik}'(s)\big)\big(\boldsymbol{\eta}_k^t(s) - \boldsymbol{\eta}_i^t(s)\big)\,ds \\
&= \int_0^\infty \mu_{ik}(s) \bigl(\boldsymbol{\eta}_k^t(s) - \boldsymbol{\eta}_i^t(s)\bigr)\, ds,
\end{align*}
}
which completes the proof.
\end{proof}

\begin{lemma}[Second-Order History Transformation]\label{lem:uncoupled_transform_2nd}
For $i \in [N]$ and $t \ge 0$, define the velocity past-history variable:
\begin{equation} \label{eq:xi_def}
\boldsymbol{\xi}_i^t(s) := \int_{t-s}^{t} \boldsymbol{v}_i(r)\,dr, \qquad s \ge 0.
\end{equation}
Then $\boldsymbol{\xi}_i^t(0) = \mathbf{0}$, $\partial_s \boldsymbol{\xi}_i^t(s) = \boldsymbol{v}_i(t-s)$, and the transport identity holds:
\begin{equation}\label{eq:transport_xi}
\partial_t \boldsymbol{\xi}_i^t(s) = \boldsymbol{v}_i(t) - \boldsymbol{v}_i(t-s) = \boldsymbol{v}_i(t) - \partial_s \boldsymbol{\xi}_i^t(s).
\end{equation}
Furthermore, the uncoupled velocity memory integral transforms into:
\begin{equation}\label{eq:ibp_velocity}
\int_{0}^{\infty} g_{ik}(s)\bigl(\boldsymbol{v}_k(t-s)-\boldsymbol{v}_i(t-s)\bigr)\,ds = \int_0^\infty \mu_{ik}(s)\bigl(\boldsymbol{\xi}_k^t(s)-\boldsymbol{\xi}_i^t(s)\bigr)\,ds.
\end{equation}
\end{lemma}

\begin{proof}
Differentiating \eqref{eq:xi_def} with respect to $t$ gives \eqref{eq:transport_xi}. Applying integration by parts in $s$ analogously to Lemma \ref{lem:uncoupled_transform_1st} gives \eqref{eq:ibp_velocity}.
\end{proof}

\section{Description of Main Results}\label{sec:main_result}
In this section, we summarize our main theoretical results concerning the global well-posedness, complete state consensus, and asymptotic flocking for the multi-agent systems \eqref{eq:model1}--\eqref{eq:model4} with infinite fading memory. Precisely, we divide our results into three parts: the global existence and uniqueness of classical solutions, the exponential consensus of first-order opinion dynamics, and the asymptotic flocking of second-order velocity alignment kinematics. The complete proofs are presented in detail in Sections \ref{sec:proof_wp}--\ref{sec:proof_2nd}.

\subsection{Global Well-Posedness}
We first address the global-in-time solvability of the Cauchy problems associated with Models I--II. For first-order systems \eqref{eq:model1} and \eqref{eq:model2}, the state is prescribed by a bounded continuous initial history $\boldsymbol{\phi} \in L^\infty(-\infty, 0]; (\mathbb{R}^n)^N)$. By establishing local contraction mappings and combining them with uniform a priori Gr\"onwall estimates, we obtain the following global well-posedness theorem.

\begin{theorem}\label{thm:wellposed}
For any given bounded continuous initial history $\boldsymbol{x}^{in} \in L^\infty(-\infty, 0]; (\mathbb{R}^n)^N)$ for first-order models, there exists a unique global classical solution $\boldsymbol{x} \in C^1([0,\infty); (\mathbb{R}^n)^N)$ to Models I and II.
\end{theorem}

For second-order systems \eqref{eq:model3} and \eqref{eq:model4}, {the initial data consist of a pointwise initial position $\boldsymbol{x}^0 \in (\mathbb{R}^n)^N$ and a bounded continuous initial velocity history $\boldsymbol{v}^{in} \in L^\infty(-\infty, 0]; (\mathbb{R}^n)^N)$}. The global well-posedness theorem is stated as follows.

\begin{theorem}\label{thm:wellposed-1}
{For any given pointwise initial position $\boldsymbol{x}^0 \in (\mathbb{R}^n)^N$ and bounded continuous initial velocity history $\boldsymbol{v}^{in} \in L^\infty(-\infty, 0]; (\mathbb{R}^n)^N)$ for second-order models}, there exists a unique global classical solution $(\boldsymbol{x},\boldsymbol{v}) \in C^1([0,\infty); (\mathbb{R}^n)^N \times (\mathbb{R}^n)^N)$ to Models~III and IV.
\end{theorem}

\subsection{Consensus for First-Order Systems}
In this subsection, we present the asymptotic consensus results for the first-order opinion dynamics \eqref{eq:model1} and \eqref{eq:model2}. We first state the exponential consensus result for Model II \eqref{eq:model2}, where instantaneous communication forces and distributed fading memory act simultaneously.

\begin{theorem}[Unconditional Consensus for Model II]\label{thm:model2}
Let $\boldsymbol{x}(t)$ be the global classical solution to Model II \eqref{eq:model2}. Then $\boldsymbol{x}(t)$ achieves global exponential consensus unconditionally, i.e., there exist positive constants $C_0$ and $\gamma$ such that
\begin{equation*}
X_{\max}(t) \le C_0 e^{-\gamma t}, \qquad \forall t \ge 0.
\end{equation*}
\end{theorem}

When the instantaneous communication force is absent ($\psi \equiv 0$), Model II \eqref{eq:model2} degenerates into Model I \eqref{eq:model1}. In this regime, system \eqref{eq:model1} lacks instantaneous state dissipation. By constructing an auxiliary position multiplier functional $\Phi(t)$ to extract the hidden dissipation from the fading memory kernel, we establish that memory alone is sufficient to guarantee exponential consensus as the following theorem.

\begin{theorem}[Unconditional Consensus for Model I]\label{thm:model1}
Let $\boldsymbol{x}(t)$ be the global classical solution to Model I \eqref{eq:model1}. Then $\boldsymbol{x}(t)$ achieves global exponential consensus unconditionally, demonstrating that fading memory intrinsically induces sufficient dissipative mechanisms without instantaneous interactions.
\end{theorem}

%

\begin{remark}\label{rem:consensus_discussion}
{\rm
Theorems \ref{thm:model2} and \ref{thm:model1} reveal that distributed fading memory acts as an intrinsic globally stabilizing dissipative mechanism in multi-agent systems.

 In the absence of memory, establishing consensus in state-dependent or switching multi-agent networks fundamentally relies on persistent dynamic connectivity conditions, such as scrambling graph structures, uniform spanning trees, or persistent chain connections over time intervals \cite{Blondel2009, Canuto2012, Su2017}. In Model I, even when instantaneous communication is entirely absent ($\psi \equiv 0$), the kernel derivative condition $\mu'_{ik}(s) \le -\delta \mu_{ik}(s)$ supplies the necessary hidden dissipation to drive the whole network to consensus. {However, we emphasize that this memory-induced dissipation fundamentally requires the static connectivity of the interaction graph, which is ensured in our framework by the nondegeneracy condition $\lambda_{\min} = \min\limits_{i \ne k} g_{ik}(0) > 0$ in Assumption \ref{assum:g}(i). Provided that this underlying memory network connects all distinct agents, consensus emerges unconditionally without requiring any additional time-varying or switching topological connectivity assumptions.}
}
\end{remark}

\subsection{Flocking for Second-Order Systems}
In this subsection, we present the asymptotic flocking results for the second-order velocity alignment systems \eqref{eq:model3} and \eqref{eq:model4}. 

{For Model~IV \eqref{eq:model4}, where instantaneous spatial interaction $\psi$ is coupled with velocity fading memory, we develop a unified energy-multiplier framework directly on the coupled system in the zero-mean velocity frame. By demonstrating that the cross-terms induced by instantaneous interactions are effectively absorbed by the instantaneous and memory dissipations, we establish the following unconditional flocking theorem.
}

\begin{theorem}[{Unconditional Flocking for Model IV}]\label{thm:model4}
{Let $(\boldsymbol{x}(t),\boldsymbol{v}(t))$ be the global classical solution to Model IV \eqref{eq:model4}. Then, the system achieves global exponential flocking unconditionally for arbitrary initial data $(\boldsymbol{x}^0, \boldsymbol{v}^{in}) \in (\mathbb{R}^n)^N \times L^\infty(-\infty, 0; (\mathbb{R}^n)^N)$, without any capacity constraint on initial configurations. Specifically, there exist positive constants $C_v > 0$ and $\gamma_v > 0$ such that
\begin{equation}\label{eq:flock_exp_rate}
V_{\max}(t) \le C_v e^{-\frac{\gamma_v}{2} t}, \qquad \forall t \ge 0,
\end{equation}
and the maximum spatial diameter remains uniformly bounded for all times:
\begin{equation}\label{eq:flock_spatial_bound}
\sup_{t \ge 0} X_{\max}(t) \le X_{\max}(0) + \frac{2C_v}{\gamma_v} < \infty.
\end{equation}
}
\end{theorem}

For Model~III \eqref{eq:model3}, where instantaneous spatial coupling is absent ($\psi \equiv 0$), the interaction is purely velocity-driven. {Since the velocity equations decouple from the positions and are mathematically isomorphic to the first-order pure memory model \eqref{eq:model1}, exponential velocity alignment follows directly from Theorem~\ref{thm:model1}, ensuring that the spatial configuration remains uniformly bounded globally without any constraint on the initial data.}

\begin{theorem}[Unconditional Flocking for Model III]\label{thm:model3}
Let $(\boldsymbol{x}(t),\boldsymbol{v}(t))$ be the global classical solution to Model III \eqref{eq:model3}. Then, without any restriction on the network size $N$ or kernel mass, the system globally achieves exponential velocity alignment, and the maximum spatial diameter remains uniformly bounded:
$$
\sup_{t\ge0} X_{\max}(t) < \infty.
$$
\end{theorem}

\begin{remark}\label{rem:flocking_discussion}
{\rm
{We highlight the physical and mathematical implications of Theorem~\ref{thm:model4}:
\begin{enumerate}
    \item [(i)] \textbf{Elimination of the Conditional Flocking Barrier:} In the classical Cucker--Smale model without memory \cite{Cucker2007, Ha2009}, flocking under short-range communication ($\int^\infty \psi(r)\,dr < \infty$) is conditional, requiring initial velocities to be sufficiently small to prevent group dispersion. In contrast, Theorem~\ref{thm:model4} demonstrates that the presence of infinite fading memory overcomes this dispersion risk: the velocity dissipation induced by the memory kernel is independent of the spatial positions, guaranteeing global flocking unconditionally even for rapidly decaying communication weights.
    \item [(ii)] \textbf{Role of Instantaneous Interaction in Model IV:} Although fading memory alone suffices for exponential flocking (as seen in Model~III), the presence of instantaneous interaction $\psi$ provides additional damping that strictly assists the alignment process without requiring any initial capacity threshold.
\end{enumerate}
}
}
\end{remark}

\section{Global Well-Posedness for both first and second order systems}\label{sec:proof_wp}
In this section, we provide the proof of Theorem \ref{thm:wellposed} and Theorem \ref{thm:wellposed-1} regarding the global-in-time existence and uniqueness of classical solutions. We present the analysis by considering the first-order opinion dynamics and the second-order flocking kinematics separately. Our general strategy relies on rewriting the integro-differential systems into equivalent Volterra-type integral equations. We first establish local well-posedness on a short time interval via the Banach Fixed-Point Theorem. Subsequently, we derive uniform a priori estimates on the state variables using Gr\"onwall's inequality to prevent finite-time blow-up, thereby extending the local solutions globally.

\noindent {\bf Proof of Theorem \ref{thm:wellposed}.} Setting
$$
X^{in} := \max_{i\in[N]} \sup_{s \le 0} \|\boldsymbol{x}^{in}_i(s)\| < \infty,
$$
we rewrite Model II \eqref{eq:model2} for $t > 0$ by splitting the memory integral over $s\in[0,t]$ and $s\in[t,\infty)$. Substituting $\tau = t-s$ on $[0,t]$ and $u = s-t$ on $[0,\infty)$ gives:
\begin{equation}\label{eq:wp_split_1st}
\begin{aligned}
\dot{\boldsymbol{x}}_i(t) &= \sum_{k=1}^N \psi(\|\boldsymbol{x}_i(t)-\boldsymbol{x}_k(t)\|) (\boldsymbol{x}_k(t)-\boldsymbol{x}_i(t)) + \sum_{k=1}^N \int_0^t g_{ik}(t-\tau)(\boldsymbol{x}_k(\tau)-\boldsymbol{x}_i(\tau)) \, d\tau \\
&\quad + \sum_{k=1}^N \int_0^\infty g_{ik}(t+u)(\boldsymbol{x}^{in}_k(-u)-\boldsymbol{x}^{in}_i(-u)) \, du.
\end{aligned}
\end{equation}
The third term is a deterministic forcing term $\boldsymbol{H}_{ik}(t) := \int_0^\infty g_{ik}(t+u)(\boldsymbol{x}^{in}_k(-u)-\boldsymbol{x}^{in}_i(-u)) \, du$, whose norm is bounded by
\begin{equation}\label{eq:H_bound_1st}
\|\boldsymbol{H}_{ik}(t)\| \le 2X^{in} \int_0^\infty g_{ik}(t+u)\,du \le 2X^{in} \kappa_{ik} \le C_H := 2X^{in} \kappa_{\max} < \infty.
\end{equation}

Integrating \eqref{eq:wp_split_1st} over $[0,t]$, we define the operator $\mathcal{L}$ acting on $\boldsymbol{x}\in C([0,T];(\mathbb{R}^n)^N)$ by
\begin{equation}\label{eq:operatorL_1st}
\begin{aligned}
[\mathcal{L}(\boldsymbol{x})]_i(t) &= \boldsymbol{x}^{in}_i(0) + \int_0^t \sum_{k=1}^N \psi(\|\boldsymbol{x}_i(r)-\boldsymbol{x}_k(r)\|)(\boldsymbol{x}_k(r)-\boldsymbol{x}_i(r)) \, dr \\
&\quad + \int_0^t \sum_{k=1}^N \int_0^r g_{ik}(r-\tau)(\boldsymbol{x}_k(\tau)-\boldsymbol{x}_i(\tau)) \, d\tau\, dr + \int_0^t \sum_{k=1}^N \boldsymbol{H}_{ik}(r) \, dr.
\end{aligned}
\end{equation}

Define $F(\boldsymbol{x}_i,\boldsymbol{x}_k):=\psi(\|\boldsymbol{x}_i-\boldsymbol{x}_k\|)(\boldsymbol{x}_k-\boldsymbol{x}_i)$. For any state vectors with
$$
\|\boldsymbol{x}_i\|,\|\boldsymbol{x}_k\|,\|\tilde{\boldsymbol{x}}_i\|,\|\tilde{\boldsymbol{x}}_k\|\le R,
$$
using the global Lipschitz continuity of $\psi$ with constant $[\psi]_{\mathrm{Lip}}$, we obtain
\begin{equation*}\label{eq:FA_lip}
\|F(\boldsymbol{x}_i,\boldsymbol{x}_k)-F(\tilde{\boldsymbol{x}}_i,\tilde{\boldsymbol{x}}_k)\| \le L_R\big(\|\boldsymbol{x}_i-\tilde{\boldsymbol{x}}_i\|+\|\boldsymbol{x}_k-\tilde{\boldsymbol{x}}_k\|\big),
\end{equation*}
where $L_R := 2R [\psi]_{\mathrm{Lip}} + \psi(0)$.

Define the Banach space $\mathcal{X}_T = C([0,T];(\mathbb{R}^n)^N)$ equipped with
$$
\|\boldsymbol{x}\|_{\mathcal{X}_T} := \sup_{t\in[0,T]}\max_{i\in[N]}\|\boldsymbol{x}_i(t)\|.
$$
Let $R := 2X^{in}$ and $B_R := \{\boldsymbol{x}\in\mathcal{X}_T: \|\boldsymbol{x}\|_{\mathcal{X}_T} \le R\}$.
For any $\boldsymbol{x} \in B_R$, we have $\|\boldsymbol{x}_k(\tau)-\boldsymbol{x}_i(\tau)\| \le 2R$. From \eqref{eq:operatorL_1st} and \eqref{eq:H_bound_1st},
\begin{align*}
\|[\mathcal{L}(\boldsymbol{x})]_i(t)\| &\le X^{in} + \int_0^t \sum_{k=1}^N \psi(0)(2R)\,dr + \int_0^t \sum_{k=1}^N \kappa_{\max}(2R)\,dr + \int_0^t \sum_{k=1}^N C_H\,dr \\
&\le X^{in} + t N \Big[ 2R \big(\psi(0)+\kappa_{\max}\big) + C_H \Big] =: X^{in} + t K_1(R).
\end{align*}
Choosing $T_1 := \frac{X^{in}}{K_1(R)}$ ensures $\mathcal{L}(B_R) \subset B_R$ for $T \le T_1$. Furthermore, for $\boldsymbol{x},\tilde{\boldsymbol{x}} \in B_R$,
{\small
\begin{align*}
\|[\mathcal{L}(\boldsymbol{x})]_i(t) - [\mathcal{L}(\tilde{\boldsymbol{x}})]_i(t)\| &\le \int_0^t \sum_{k=1}^N L_R\big(\|\boldsymbol{x}_i(r)-\tilde{\boldsymbol{x}}_i(r)\|+\|\boldsymbol{x}_k(r)-\tilde{\boldsymbol{x}}_k(r)\|\big)\,dr \\
&\quad + \int_0^t \sum_{k=1}^N \int_0^r g_{ik}(r-\tau) \big(\|\boldsymbol{x}_i(\tau)-\tilde{\boldsymbol{x}}_i(\tau)\|+\|\boldsymbol{x}_k(\tau)-\tilde{\boldsymbol{x}}_k(\tau)\|\big)\,d\tau\,dr \\
&\le T \cdot 2 N (L_R + \kappa_{\max}) \|\boldsymbol{x} - \tilde{\boldsymbol{x}}\|_{\mathcal{X}_T}.
\end{align*}
}
Setting $T_2 := \frac{1}{4 N (L_R + \kappa_{\max})}$, the operator $\mathcal{L}$ is a strict contraction on $B_R$ for $T := \min\{T_1, T_2\}$. By the Banach Fixed-Point Theorem, a unique local solution exists on $[0,T]$.

To extend the solution globally, define
$$
M(t) := \sup_{r\in[0,t]}\max_{i\in[N]}\|\boldsymbol{x}_i(r)\|.
$$
Along the local solution,
\begin{equation*}
M(t) \le X^{in} + N C_H t + \int_0^t 2 N \big(\psi(0)+\kappa_{\max}\big) M(\tau)\,d\tau.
\end{equation*}
By Gr\"onwall's inequality, for all $t \ge 0,$
\begin{equation*}\label{eq:gronwall_bound_1st}
M(t) \le \big(X^{in} + N C_H t\big) e^{2 N (\psi(0)+\kappa_{\max}) t} < \infty.
\end{equation*}
Since $M(t)$ cannot experience finite-time blow-up, standard continuation arguments confirm that $T_{\max} = \infty$. Model I \eqref{eq:model1} follows directly as a special case by setting $\psi \equiv 0$.

\noindent {\bf Proof of Theorem \ref{thm:wellposed-1}.} For Model IV \eqref{eq:model4}, let
$$
X^0 := \max_{i\in[N]}  \|\boldsymbol{x}^{0}_i\| \quad \text{and} \quad V^{in} := \max_{i\in[N]} \sup_{s \le 0} \|\boldsymbol{v}^{in}_i(s)\| < \infty.
$$
For $t > 0$, splitting the memory integral over $s\in[0,t]$ and $s\in[t,\infty)$ yields:
{\footnotesize
\begin{equation}\label{eq:wp_split_2nd}
\left\{
\begin{aligned}
\dot{\boldsymbol{x}}_i(t) &= \boldsymbol{v}_i(t),  \\
\dot{\boldsymbol{v}}_i(t) &= \sum_{k=1}^N \psi(\|\boldsymbol{x}_i(t)-\boldsymbol{x}_k(t)\|) (\boldsymbol{v}_k(t)-\boldsymbol{v}_i(t)) + \sum_{k=1}^N \int_0^t g_{ik}(t-\tau)(\boldsymbol{v}_k(\tau)-\boldsymbol{v}_i(\tau)) \, d\tau + \sum_{k=1}^N \boldsymbol{G}_{ik}(t),
\end{aligned}
\right.
\end{equation}
}
where $\boldsymbol{G}_{ik}(t) := \int_0^\infty g_{ik}(t+u)(\boldsymbol{v}^{in}_k(-u)-\boldsymbol{v}^{in}_i(-u)) \, du$ satisfies
$$
\|\boldsymbol{G}_{ik}(t)\| \le 2V^{in} \kappa_{\max} =: C_G < \infty.
$$

Integrating \eqref{eq:wp_split_2nd} over $[0,t]$, we define the solution operator $\mathcal{T}(\boldsymbol{x},\boldsymbol{v}) := (\mathcal{T}_x(\boldsymbol{x},\boldsymbol{v}), \mathcal{T}_v(\boldsymbol{x},\boldsymbol{v}))$ on $C([0,T];(\mathbb{R}^n)^N \times (\mathbb{R}^n)^N)$:
\begin{align*}
[\mathcal{T}_x(\boldsymbol{x},\boldsymbol{v})]_i(t) &= \boldsymbol{x}^0_i + \int_0^t \boldsymbol{v}_i(r)\,dr, \\
[\mathcal{T}_v(\boldsymbol{x},\boldsymbol{v})]_i(t) &= \boldsymbol{v}^{in}_i(0) + \int_0^t \sum_{k=1}^N \psi(\|\boldsymbol{x}_i(r)-\boldsymbol{x}_k(r)\|)(\boldsymbol{v}_k(r)-\boldsymbol{v}_i(r)) \, dr \\
&\quad + \int_0^t \sum_{k=1}^N \int_0^r g_{ik}(r-\tau)(\boldsymbol{v}_k(\tau)-\boldsymbol{v}_i(\tau)) \, d\tau\, dr + \int_0^t \sum_{k=1}^N \boldsymbol{G}_{ik}(r) \, dr.
\end{align*}
Define the product metric space $\mathcal{Y}_T = C([0,T];(\mathbb{R}^n)^N \times (\mathbb{R}^n)^N)$ with norm
$$
\|(\boldsymbol{x},\boldsymbol{v})\|_{\mathcal{Y}_T} := \sup_{t\in[0,T]}\max_{i\in[N]} \big( \|\boldsymbol{x}_i(t)\| + \|\boldsymbol{v}_i(t)\| \big).
$$
Setting $R := 2(X^0 + V^{in})$ and $B_R := \{(\boldsymbol{x},\boldsymbol{v}) \in \mathcal{Y}_T : \|(\boldsymbol{x},\boldsymbol{v})\|_{\mathcal{Y}_T} \le R\}$, the local Lipschitz estimates on $\psi$ and integral bounds guarantee that $\mathcal{T}$ is a strict self-mapping and contraction on $B_R$ for a sufficiently small local time horizon $T > 0$.

For global continuation, defining
$$
M_v(t) := \sup_{r\in[0,t]}\max_{i\in[N]}\|\boldsymbol{v}_i(r)\|,
$$
we derive
\begin{equation*}
M_v(t) \le V^{in} + N C_G t + \int_0^t 2 N \big(\psi(0)+\kappa_{\max}\big) M_v(\tau)\,d\tau.
\end{equation*}
By Gr\"onwall's inequality,  for all $t \ge 0,$
$$
M_v(t) \le \big(V^{in} + N C_G t\big) e^{2 N (\psi(0)+\kappa_{\max}) t} < \infty.
$$
Consequently,
$$
M_x(t) := \sup_{r\in[0,t]}\max_{i\in[N]}\|\boldsymbol{x}_i(r)\| \le X^0 + \int_0^t M_v(\tau)\,d\tau < \infty.
$$
This prevents finite-time blow-up and proves that $T_{\max} = \infty$. Model III \eqref{eq:model3} follows as a special case with $\psi \equiv 0$.

\section{Exponential Consensus for First-Order Systems}\label{sec:proof_1st}

This section is devoted to the asymptotic consensus analysis of the first-order opinion dynamics, completing the proofs of Theorem \ref{thm:model2} and Theorem \ref{thm:model1}. We begin by analyzing Model II, where the coexistence of instantaneous communication and fading memory naturally provides sufficient state dissipation, which can be directly extracted using the standard Dafermos energy functional. In contrast, Model I lacks instantaneous damping. To uncover the hidden dissipation induced exclusively by the fading memory kernel, we introduce an auxiliary cross-multiplier functional. By combining this multiplier with the standard energy, we construct a  Lyapunov functional to prove unconditional exponential consensus.

\subsection{Exponential Consensus for system \eqref{eq:model2}}
By Lemma \ref{lem:uncoupled_transform_1st}, Model II \eqref{eq:model2} maps into the autonomous extended system:
\begin{equation}\label{eq:sys_auto_model2}
\dot{\boldsymbol{x}}_i(t) = \sum_{k=1}^N \psi(\|\boldsymbol{x}_i(t)-\boldsymbol{x}_k(t)\|)\big(\boldsymbol{x}_k(t) - \boldsymbol{x}_i(t)\big) + \sum_{k=1}^N\int_0^\infty \mu_{ik}(s)\big(\boldsymbol{\eta}_k^t(s) - \boldsymbol{\eta}_i^t(s)\big)\,ds.
\end{equation}
We define the first-order position energy functional $E_0(t)$:
\begin{equation*}\label{eq:E0_def}
E_0(t) := \frac{1}{2}\sum_{i=1}^N\|\boldsymbol{x}_i(t)\|^2 + \frac{1}{4}\int_0^\infty\sum_{i,k=1}^N \mu_{ik}(s)\|\boldsymbol{\eta}_i^t(s)-\boldsymbol{\eta}_k^t(s)\|^2\,ds.
\end{equation*}

\noindent {\bf Proof of Theorem \ref{thm:model2}.} Summing \eqref{eq:sys_auto_model2} over $i \in [N]$ gives $\frac{d}{dt}\sum\limits_{i=1}^N \boldsymbol{x}_i(t) = \mathbf{0}$. By Galilean invariance, we assume without loss of generality that $\sum\limits_{i=1}^N \boldsymbol{x}_i(0) = \mathbf{0}$, whence $\sum\limits_{i=1}^N \boldsymbol{x}_i(t) \equiv \mathbf{0}$ for all $t \ge 0$.

Differentiating the state component along \eqref{eq:sys_auto_model2} and symmetrically exchanging indices $i \leftrightarrow k$ yields:
\begin{align}\label{eq:E0_dot_state}
\frac{d}{dt}\left( \frac{1}{2}\sum_{i=1}^N\|\boldsymbol{x}_i(t)\|^2 \right) &= \sum_{i=1}^N \boldsymbol{x}_i(t)^T \dot{\boldsymbol{x}}_i(t) \notag \\
&= \sum_{i=1}^N \boldsymbol{x}_i(t)^T \Biggl( \sum_{k=1}^N \psi(\|\boldsymbol{x}_i(t)-\boldsymbol{x}_k(t)\|)\big(\boldsymbol{x}_k(t) - \boldsymbol{x}_i(t)\big) \notag \\
&\qquad\qquad\qquad + \sum_{k=1}^N\int_0^\infty \mu_{ik}(s)\big(\boldsymbol{\eta}_k^t(s) - \boldsymbol{\eta}_i^t(s)\big)\,ds \Biggr) \notag \\
&= -\frac{1}{2}\sum_{i,k=1}^N \psi(\|\boldsymbol{x}_i(t)-\boldsymbol{x}_k(t)\|)\|\boldsymbol{x}_i(t)-\boldsymbol{x}_k(t)\|^2 \notag \\
&\quad -\frac{1}{2}\int_0^\infty\sum_{i,k=1}^N \mu_{ik}(s)(\boldsymbol{x}_i(t)-\boldsymbol{x}_k(t))^T(\boldsymbol{\eta}_i^t(s)-\boldsymbol{\eta}_k^t(s))\,ds.
\end{align}

Differentiating the memory energy component utilizing the transport identity \eqref{eq:transport_uncoupled}:
{\footnotesize
\begin{align}\label{eq:E0_dot_memory}
&\frac{d}{dt} \left( \frac{1}{4}\int_0^\infty\sum_{i,k=1}^N \mu_{ik}(s)\|\boldsymbol{\eta}_i^t(s)-\boldsymbol{\eta}_k^t(s)\|^2\,ds \right) \notag \\
&= \frac{1}{2}\int_0^\infty\sum_{i,k=1}^N\mu_{ik}(s)(\boldsymbol{\eta}_i^t(s)-\boldsymbol{\eta}_k^t(s))^T\bigl( \partial_t \boldsymbol{\eta}_i^t(s) - \partial_t \boldsymbol{\eta}_k^t(s) \bigr)\,ds \notag \\
&= \frac{1}{2}\int_0^\infty\sum_{i,k=1}^N\mu_{ik}(s)(\boldsymbol{\eta}_i^t(s)-\boldsymbol{\eta}_k^t(s))^T\bigl[ (\boldsymbol{x}_i(t)-\boldsymbol{x}_k(t)) - \partial_s (\boldsymbol{\eta}_i^t(s)-\boldsymbol{\eta}_k^t(s)) \bigr]\,ds \notag \\
&= \frac{1}{2}\int_0^\infty\sum_{i,k=1}^N\mu_{ik}(s)(\boldsymbol{x}_i(t)-\boldsymbol{x}_k(t))^T(\boldsymbol{\eta}_i^t(s)-\boldsymbol{\eta}_k^t(s))\,ds - \frac{1}{4}\int_0^\infty\sum_{i,k=1}^N\mu_{ik}(s)\partial_s\|\boldsymbol{\eta}_i^t(s)-\boldsymbol{\eta}_k^t(s)\|^2\,ds.
\end{align}
}

Summing \eqref{eq:E0_dot_state} and \eqref{eq:E0_dot_memory}, the cross-terms cancel identically. Applying integration by parts to the remaining $\partial_s$ term with boundary conditions $\boldsymbol{\eta}_i^t(0)=\mathbf{0}$ and $\lim\limits_{t\rightarrow\infty}\mu_{ik}(t)=0$:
{\small
\begin{align}\label{eq:dot_E0_final}
\dot{E}_0(t) &= -\frac{1}{2}\sum_{i,k=1}^N \psi(\|\boldsymbol{x}_i(t)-\boldsymbol{x}_k(t)\|)\|\boldsymbol{x}_i(t)-\boldsymbol{x}_k(t)\|^2 \notag \\
&\quad - \frac{1}{4}\sum_{i,k=1}^N \left[ \mu_{ik}(s) \|\boldsymbol{\eta}_i^t(s) - \boldsymbol{\eta}_k^t(s)\|^2 \right]_0^\infty + \frac{1}{4}\int_0^\infty\sum_{i,k=1}^N \mu'_{ik}(s)\|\boldsymbol{\eta}_i^t(s)-\boldsymbol{\eta}_k^t(s)\|^2\,ds \notag \\
&= -\frac{1}{2}\sum_{i,k=1}^N \psi(\|\boldsymbol{x}_i(t)-\boldsymbol{x}_k(t)\|)\|\boldsymbol{x}_i(t)-\boldsymbol{x}_k(t)\|^2 + \frac{1}{4}\int_0^\infty\sum_{i,k=1}^N \mu'_{ik}(s)\|\boldsymbol{\eta}_i^t(s)-\boldsymbol{\eta}_k^t(s)\|^2\,ds \le 0.
\end{align}
}

Since $E_0(t) \le E_0(0)$ for all $t \ge 0$, the maximum spatial diameter is uniformly bounded:
\begin{equation*}
X_{\max}(t) \le 2 \max_{i\in[N]} \|\boldsymbol{x}_i(t)\| \le 2\sqrt{2E_0(0)} =: R_{\max}, \qquad \forall t \ge 0.
\end{equation*}
By Assumption \ref{assum:psi}, $\psi(\|\boldsymbol{x}_i(t)-\boldsymbol{x}_k(t)\|) \ge \psi(R_{\max}) > 0$. Using $\mu'_{ik}(s) \le -\delta \mu_{ik}(s)$ from Assumption \ref{assum:g} together with the algebraic identity $\sum\limits_{i,k=1}^N \|\boldsymbol{x}_i(t)-\boldsymbol{x}_k(t)\|^2 = 2N \sum\limits_{i=1}^N \|\boldsymbol{x}_i(t)\|^2$, we obtain
\begin{align*}
\dot{E}_0(t) &\le -\psi(R_{\max}) N \sum_{i=1}^N\|\boldsymbol{x}_i(t)\|^2 - \frac{\delta}{4}\int_0^\infty\sum_{i,k=1}^N \mu_{ik}(s)\|\boldsymbol{\eta}_i^t(s)-\boldsymbol{\eta}_k^t(s)\|^2\,ds \le -\gamma E_0(t),
\end{align*}
where $\gamma := \min\{2N\psi(R_{\max}), \delta\} > 0$. Applying Gr\"onwall's Lemma gives $E_0(t) \le E_0(0)e^{-\gamma t}$, which proves unconditional exponential consensus for Model II.

\subsection{Exponential Consensus for system \eqref{eq:model1}} For Model I ($\psi \equiv 0$), equation \eqref{eq:dot_E0_final} reduces to
$$
\dot{E}_0(t) = \frac{1}{4}\int_0^\infty\sum_{i,k=1}^N \mu'_{ik}(s)\|\boldsymbol{\eta}_i^t(s)-\boldsymbol{\eta}_k^t(s)\|^2\,ds \le 0,
$$
which lacks instantaneous state dissipation. We introduce the auxiliary multiplier functional:
\begin{equation}\label{eq:multiplier_M1}
\Phi(t) := -\frac{1}{2}\int_0^\infty\sum_{i,k=1}^N\mu_{ik}(s)\big(\boldsymbol{x}_i(t)-\boldsymbol{x}_k(t)\big)^T\big(\boldsymbol{\eta}_i^t(s)-\boldsymbol{\eta}_k^t(s)\big)\,ds.
\end{equation}

\noindent {\bf Proof of Theorem \ref{thm:model1}.}
Differentiating \eqref{eq:multiplier_M1} along \eqref{eq:sys_auto_model2} (with $\psi \equiv 0$) using $\partial_t \boldsymbol{\eta}_i^t(s) = \boldsymbol{x}_i(t) - \partial_s \boldsymbol{\eta}_i^t(s)$:
\begin{align}\label{eq:Phi_dot_split}
\dot{\Phi}(t) &= -\frac{1}{2}\int_0^\infty\sum_{i,k=1}^N\mu_{ik}(s)\big(\dot{\boldsymbol{x}}_i(t)-\dot{\boldsymbol{x}}_k(t)\big)^T\big(\boldsymbol{\eta}_i^t(s)-\boldsymbol{\eta}_k^t(s)\big)\,ds \notag \\
&\quad - \frac{1}{2}\int_0^\infty\sum_{i,k=1}^N\mu_{ik}(s)\big(\boldsymbol{x}_i(t)-\boldsymbol{x}_k(t)\big)^T\big[\big(\boldsymbol{x}_i(t)-\boldsymbol{x}_k(t)\big) - \partial_s\big(\boldsymbol{\eta}_i^t(s)-\boldsymbol{\eta}_k^t(s)\big)\big]\,ds.
\end{align}

Exchanging indices $i \leftrightarrow k$ and applying the definition
$$
\dot{\boldsymbol{x}}_i(t) = \sum_{k=1}^N \int_0^\infty \mu_{ik}(s)(\boldsymbol{\eta}_k^t(s)-\boldsymbol{\eta}_i^t(s))\,ds,
$$
the first term expands to:
{\small
\begin{align}\label{eq:Phi_term1}
&-\frac{1}{2}\int_0^\infty\sum_{i,k=1}^N \mu_{ik}(s)\big(\dot{\boldsymbol{x}}_i(t)-\dot{\boldsymbol{x}}_k(t)\big)^T\big(\boldsymbol{\eta}_i^t(s)-\boldsymbol{\eta}_k^t(s)\big)\,ds \notag \\
&= \sum_{i=1}^N \dot{\boldsymbol{x}}_i(t)^T \left( -\sum_{k=1}^N \int_0^\infty \mu_{ik}(s)\big(\boldsymbol{\eta}_i^t(s)-\boldsymbol{\eta}_k^t(s)\big)\,ds \right) = \sum_{i=1}^N \dot{\boldsymbol{x}}_i(t)^T \dot{\boldsymbol{x}}_i(t) = \sum_{i=1}^N \|\dot{\boldsymbol{x}}_i(t)\|^2.
\end{align}
}

For the second term in \eqref{eq:Phi_dot_split}, splitting the integrand and applying integration by parts on $\partial_s(\boldsymbol{\eta}_i^t(s)-\boldsymbol{\eta}_k^t(s))$:
{\small
\begin{align}\label{eq:Phi_term2}
&-\frac{1}{2}\sum_{i,k=1}^N \left(\int_0^\infty \mu_{ik}(s)\,ds\right) \|\boldsymbol{x}_i(t)-\boldsymbol{x}_k(t)\|^2 \notag \\
&\quad + \frac{1}{2}\int_0^\infty \sum_{i,k=1}^N \mu_{ik}(s) (\boldsymbol{x}_i(t)-\boldsymbol{x}_k(t))^T \partial_s (\boldsymbol{\eta}_i^t(s)-\boldsymbol{\eta}_k^t(s))\,ds \notag \\
&= -\frac{1}{2}\sum_{i,k=1}^N \lambda_{ik} \|\boldsymbol{x}_i(t)-\boldsymbol{x}_k(t)\|^2 + \frac{1}{2}\sum_{i,k=1}^N (\boldsymbol{x}_i(t)-\boldsymbol{x}_k(t))^T \left[ \mu_{ik}(s)(\boldsymbol{\eta}_i^t(s)-\boldsymbol{\eta}_k^t(s)) \right]_0^\infty \notag \\
&\quad - \frac{1}{2}\int_0^\infty \sum_{i,k=1}^N \mu'_{ik}(s) (\boldsymbol{x}_i(t)-\boldsymbol{x}_k(t))^T (\boldsymbol{\eta}_i^t(s)-\boldsymbol{\eta}_k^t(s))\,ds \notag \\
&= -\frac{1}{2}\sum_{i,k=1}^N \lambda_{ik} \|\boldsymbol{x}_i(t)-\boldsymbol{x}_k(t)\|^2 - \frac{1}{2}\int_0^\infty \sum_{i,k=1}^N \mu'_{ik}(s) (\boldsymbol{x}_i(t)-\boldsymbol{x}_k(t))^T (\boldsymbol{\eta}_i^t(s)-\boldsymbol{\eta}_k^t(s))\,ds,
\end{align}
}
where $\lambda_{ik} := \int_0^\infty \mu_{ik}(s)\,ds = g_{ik}(0) > 0$. Combining \eqref{eq:Phi_term1} and \eqref{eq:Phi_term2}:
{\small
\begin{equation*}\label{eq:Phi_dot_final}
\dot{\Phi}(t) = \sum_{i=1}^N \|\dot{\boldsymbol{x}}_i(t)\|^2 - \frac{1}{2}\sum_{i,k=1}^N \lambda_{ik}\|\boldsymbol{x}_i(t)-\boldsymbol{x}_k(t)\|^2 - \frac{1}{2}\int_0^\infty \sum_{i,k=1}^N \mu'_{ik}(s)(\boldsymbol{x}_i(t)-\boldsymbol{x}_k(t))^T(\boldsymbol{\eta}_i^t(s)-\boldsymbol{\eta}_k^t(s))\,ds.
\end{equation*}
}

Define the total Lyapunov functional $\mathcal{E}(t) := E_0(t) + \varepsilon \Phi(t)$ for $\varepsilon > 0$. Since $\mu'_{ik}(s) \le 0$, writing $-\mu'_{ik}(s) = |\mu'_{ik}(s)| \ge 0$, we apply Young's inequality with parameter $\alpha > 0$:
{\small
\begin{align}\label{eq:cross_bound}
&\frac{\varepsilon}{2}\int_0^\infty \sum_{i,k=1}^N (-\mu'_{ik}(s))(\boldsymbol{x}_i(t)-\boldsymbol{x}_k(t))^T(\boldsymbol{\eta}_i^t(s)-\boldsymbol{\eta}_k^t(s))\,ds \notag \\
&\le \frac{\varepsilon}{4\alpha}\int_0^\infty \sum_{i,k=1}^N (-\mu'_{ik}(s)) \|\boldsymbol{x}_i(t)-\boldsymbol{x}_k(t)\|^2\,ds + \frac{\varepsilon\alpha}{4} \int_0^\infty \sum_{i,k=1}^N (-\mu'_{ik}(s)) \|\boldsymbol{\eta}_i^t(s)-\boldsymbol{\eta}_k^t(s)\|^2\,ds \notag \\
&= \frac{\varepsilon}{4\alpha} \sum_{i,k=1}^N \mu_{ik}(0) \|\boldsymbol{x}_i(t)-\boldsymbol{x}_k(t)\|^2 + \frac{\varepsilon\alpha}{4} \int_0^\infty \sum_{i,k=1}^N (-\mu'_{ik}(s)) \|\boldsymbol{\eta}_i^t(s)-\boldsymbol{\eta}_k^t(s)\|^2\,ds,
\end{align}
}
where we used $\int_0^\infty (-\mu'_{ik}(s))\,ds = \mu_{ik}(0) - \lim\limits_{t\rightarrow\infty}\mu_{ik}(t) = \mu_{ik}(0)$.

Next, applying the Cauchy--Schwarz inequality to $\dot{\boldsymbol{x}}_i(t) = \sum\limits_{k=1}^N \int_0^\infty \mu_{ik}(s)(\boldsymbol{\eta}_k^t(s)-\boldsymbol{\eta}_i^t(s))\,ds$:
\begin{align}\label{eq:xdot_bound}
\|\dot{\boldsymbol{x}}_i(t)\|^2 &= \left\| \sum_{k=1}^N \int_0^\infty \frac{\mu_{ik}(s)}{\sqrt{-\mu'_{ik}(s)}} \cdot \sqrt{-\mu'_{ik}(s)}\, (\boldsymbol{\eta}_k^t(s)-\boldsymbol{\eta}_i^t(s))\,ds \right\|^2 \notag \\
&\le \left(\sum_{k=1}^N \int_0^\infty \frac{\mu_{ik}(s)^2}{-\mu'_{ik}(s)}\,ds\right) \left(\sum_{k=1}^N \int_0^\infty (-\mu'_{ik}(s))\|\boldsymbol{\eta}_i^t(s)-\boldsymbol{\eta}_k^t(s)\|^2\,ds\right).
\end{align}
By Assumption \ref{assum:g}, $\mu_{ik}(s) \le \frac{1}{\delta}(-\mu'_{ik}(s))$, so the first parenthesis is bounded by:
\begin{equation*}
\sum_{k=1}^N \int_0^\infty \frac{\mu_{ik}(s)^2}{-\mu'_{ik}(s)}\,ds \le \frac{1}{\delta}\sum_{k=1}^N \int_0^\infty \mu_{ik}(s)\,ds = \frac{1}{\delta}\sum_{k=1}^N \lambda_{ik} \le \frac{N\lambda_{\max}}{\delta} =: \Lambda < \infty.
\end{equation*}
Summing \eqref{eq:xdot_bound} over $i \in [N]$ yields:
\begin{equation}\label{eq:sum_xdot_bound}
\varepsilon \sum_{i=1}^N \|\dot{\boldsymbol{x}}_i(t)\|^2 \le \varepsilon \Lambda \int_0^\infty \sum_{i,k=1}^N (-\mu'_{ik}(s)) \|\boldsymbol{\eta}_i^t(s)-\boldsymbol{\eta}_k^t(s)\|^2\,ds.
\end{equation}

Substituting \eqref{eq:cross_bound} and \eqref{eq:sum_xdot_bound} into $\dot{\mathcal{E}}(t) = \dot{E}_0(t) + \varepsilon \dot{\Phi}(t)$:
\begin{align}\label{eq:E_dot_assembled}
\dot{\mathcal{E}}(t) &\le -\varepsilon \sum_{i,k=1}^N \left( \frac{\lambda_{ik}}{2} - \frac{\mu_{ik}(0)}{4\alpha} \right) \|\boldsymbol{x}_i(t)-\boldsymbol{x}_k(t)\|^2 \notag \\
&\quad - \int_0^\infty \sum_{i,k=1}^N (-\mu'_{ik}(s)) \|\boldsymbol{\eta}_i^t(s)-\boldsymbol{\eta}_k^t(s)\|^2 \left( \frac{1}{4} - \frac{\varepsilon\alpha}{4} - \varepsilon\Lambda \right)\,ds.
\end{align}

Setting $\lambda_{\min} := \min\limits_{i,k}\lambda_{ik} > 0$ and $\mu^0_{\max} := \max\limits_{i,k}\mu_{ik}(0) < \infty$, we choose $\alpha > 0$ sufficiently large such that
\begin{equation*}
\frac{\lambda_{ik}}{2} - \frac{\mu_{ik}(0)}{4\alpha} \ge \frac{\lambda_{\min}}{2} - \frac{\mu^0_{\max}}{4\alpha} =: c_x > 0, \qquad \forall i,k \in [N].
\end{equation*}
With $\alpha$ fixed, we choose $\varepsilon > 0$ sufficiently small such that
\begin{equation*}
c_\eta := \frac{1}{4} - \varepsilon\left( \frac{\alpha}{4} + \Lambda \right) > 0.
\end{equation*}

Finally, we establish topological equivalence between $\mathcal{E}(t)$ and $E_0(t)$. Applying Cauchy--Schwarz to $\Phi(t)$:
\begin{flalign*}\label{eq:Phi_equiv}
|\Phi(t)| &\le \frac{1}{2}\int_0^\infty \sum_{i,k=1}^N \mu_{ik}(s) \left( \frac{1}{2}\|\boldsymbol{x}_i(t)-\boldsymbol{x}_k(t)\|^2 + \frac{1}{2}\|\boldsymbol{\eta}_i^t(s)-\boldsymbol{\eta}_k^t(s)\|^2 \right) ds \notag \\
&= \frac{1}{4}\sum_{i,k=1}^N \lambda_{ik} \|\boldsymbol{x}_i(t)-\boldsymbol{x}_k(t)\|^2 + \frac{1}{4}\int_0^\infty \sum_{i,k=1}^N \mu_{ik}(s) \|\boldsymbol{\eta}_i^t(s)-\boldsymbol{\eta}_k^t(s)\|^2 ds \notag \\
&\le \lambda_{\max} \left( \frac{N}{2}\sum_{i=1}^N \|\boldsymbol{x}_i(t)\|^2 \right) + \frac{1}{4}\int_0^\infty \sum_{i,k=1}^N \mu_{ik}(s) \|\boldsymbol{\eta}_i^t(s)-\boldsymbol{\eta}_k^t(s)\|^2 ds \notag \\
&\le \max\{N\lambda_{\max}, 1\} E_0(t) =: C_\Phi E_0(t).
\end{flalign*}
By restricting $\varepsilon \le \frac{1}{2 C_\Phi}$, we ensure $|\varepsilon \Phi(t)| \le \frac{1}{2} E_0(t)$, establishing the uniform bounds:
\begin{equation*}\label{eq:equivalence0}
\frac{1}{2} E_0(t) \le \mathcal{E}(t) \le \frac{3}{2} E_0(t), \qquad \forall t \ge 0.
\end{equation*}

Using $-\mu'_{ik}(s) \ge \delta \mu_{ik}(s)$ from Remark \ref{rem:kernel_implication}, \eqref{eq:E_dot_assembled} yields
\begin{align*}
\dot{\mathcal{E}}(t) &\le -2N{\varepsilon c_x} \sum_{i=1}^N \|\boldsymbol{x}_i(t)\|^2 - 4\delta c_\eta \left( \frac{1}{4}\int_0^\infty \sum_{i,k=1}^N \mu_{ik}(s)\|\boldsymbol{\eta}_i^t(s)-\boldsymbol{\eta}_k^t(s)\|^2\,ds \right) \notag \\
&\le -\min\{4N{\varepsilon c_x}, 4\delta c_\eta\} E_0(t) \le -\frac{2}{3}\min\{4N{\varepsilon c_x}, 4\delta c_\eta\} \mathcal{E}(t) =: -\gamma \mathcal{E}(t).
\end{align*}
By Gr\"onwall's Lemma, $\mathcal{E}(t) \le \mathcal{E}(0)e^{-\gamma t}$, and hence $E_0(t) \le 2\mathcal{E}(t) \le 3E_0(0)e^{-\gamma t}$. This establishes unconditional global exponential consensus for Model I.

\section{Exponential flocking for Second-Order Systems}\label{sec:proof_2nd}

{In this section, we establish the exponential flocking behaviors for the second-order alignment models, providing the detailed proofs for Theorem~\ref{thm:model4} and Theorem~\ref{thm:model3}. For Model~IV, where spatial communication and velocity fading memory act simultaneously, we develop a unified energy-multiplier framework directly on the coupled system in the zero-mean velocity frame. By showing that the cross-terms induced by instantaneous interactions are absorbed by the instantaneous and memory dissipations, we obtain unconditional exponential velocity alignment, from which the uniform spatial boundedness follows directly by integrating the decaying velocity bounds. For Model~III, where instantaneous coupling is absent, the velocity dynamics decouples from the positions and is isomorphic to the first-order pure memory model, allowing flocking to be deduced directly as a consequence of Theorem~\ref{thm:model1}.
}

{\subsection{Exponential flocking for system \eqref{eq:model4}}
By Lemma \ref{lem:uncoupled_transform_2nd}, Model IV can be rewritten as the autonomous extended system:
{\small
\begin{equation}\label{eq:model4_auto}
\begin{cases}
\dot{\boldsymbol{x}}_i(t) = \boldsymbol{v}_i(t), \\[4pt]
\dot{\boldsymbol{v}}_i(t) = \displaystyle\sum_{k=1}^{N} \psi(\|\boldsymbol{x}_i(t)-\boldsymbol{x}_k(t)\|)\bigl(\boldsymbol{v}_k(t)-\boldsymbol{v}_i(t)\bigr) + \sum_{k=1}^{N} \int_0^\infty \mu_{ik}(s)\bigl(\boldsymbol{\xi}_k^t(s)-\boldsymbol{\xi}_i^t(s)\bigr)\,ds.
\end{cases}
\end{equation}
}
Recalling Remark~\ref{rem:conservation_laws}, we work without loss of generality in the zero-mean velocity frame:
\[
\sum_{i=1}^N \boldsymbol{v}_i(t) \equiv \mathbf{0}, \qquad \forall t \ge 0.
\]
We define the velocity energy functional $E_v(t)$ and the auxiliary velocity multiplier functional $\Phi_v(t)$ by:
\begin{align}
E_v(t) &:= \frac12\sum_{i=1}^N\|\boldsymbol{v}_i(t)\|^2 + \frac14\int_0^\infty\sum_{i,k=1}^N \mu_{ik}(s)\|\boldsymbol{\xi}_i^t(s)-\boldsymbol{\xi}_k^t(s)\|^2\,ds, \label{eq:Ev_unified} \\
\Phi_v(t) &:= -\frac12\int_0^\infty\sum_{i,k=1}^N \mu_{ik}(s)\bigl(\boldsymbol{v}_i(t)-\boldsymbol{v}_k(t)\bigr)^T\bigl(\boldsymbol{\xi}_i^t(s)-\boldsymbol{\xi}_k^t(s)\bigr)\,ds. \label{eq:Phiv_unified}
\end{align}
For notational brevity, we define the memory force vector $\boldsymbol{m}_i(t) := \sum\limits_{k=1}^N \int_0^\infty \mu_{ik}(s)\bigl(\boldsymbol{\xi}_i^t(s)-\boldsymbol{\xi}_k^t(s)\bigr)\,ds$ and the memory dissipation functional:
\[
Q_v(t) := \sum_{i,k=1}^N \int_0^\infty (-\mu'_{ik}(s))\|\boldsymbol{\xi}_i^t(s)-\boldsymbol{\xi}_k^t(s)\|^2\,ds.
\]
}
{\noindent {\bf Proof of Theorem \ref{thm:model4}.} We divide the proof into five steps.\\
\noindent\textit{Step 1: Energy Dissipation Inequality.}
We first evaluate the time derivative of the velocity energy functional $E_v(t)$ defined in \eqref{eq:Ev_unified}. Differentiating the state velocity kinetic component along the trajectories of \eqref{eq:model4_auto} and symmetrically exchanging the indices $i \leftrightarrow k$ yields:
\begin{align}\label{eq:Ev_dot_state}
\frac{d}{dt}\left(\frac12\sum_{i=1}^N\|\boldsymbol{v}_i(t)\|^2\right) &= \sum_{i=1}^N \boldsymbol{v}_i(t)^T \dot{\boldsymbol{v}}_i(t) \notag \\
&= \sum_{i=1}^N \boldsymbol{v}_i(t)^T \Biggl( \sum_{k=1}^{N} \psi(\|\boldsymbol{x}_i(t)-\boldsymbol{x}_k(t)\|)\bigl(\boldsymbol{v}_k(t)-\boldsymbol{v}_i(t)\bigr) \notag \\
&\qquad\qquad\qquad + \sum_{k=1}^{N} \int_0^\infty \mu_{ik}(s)\bigl(\boldsymbol{\xi}_k^t(s)-\boldsymbol{\xi}_i^t(s)\bigr)\,ds \Biggr) \notag \\
&= -\frac12\sum_{i,k=1}^N\psi(\|\boldsymbol{x}_i(t)-\boldsymbol{x}_k(t)\|)\|\boldsymbol{v}_i(t)-\boldsymbol{v}_k(t)\|^2 \notag\\
&\quad -\frac12\int_0^\infty\sum_{i,k=1}^N\mu_{ik}(s)\bigl(\boldsymbol{v}_i(t)-\boldsymbol{v}_k(t)\bigr)^T\bigl(\boldsymbol{\xi}_i^t(s)-\boldsymbol{\xi}_k^t(s)\bigr)\,ds.
\end{align}
}
{Differentiating the velocity memory energy utilizing the transport identity $\partial_t\boldsymbol{\xi}_i^t(s) = \boldsymbol{v}_i(t) - \partial_s\boldsymbol{\xi}_i^t(s)$, we obtain:
\begin{align}\label{eq:Ev_dot_memory}
&\frac{d}{dt}\left(\frac14\int_0^\infty\sum_{i,k=1}^N\mu_{ik}(s)\|\boldsymbol{\xi}_i^t(s)-\boldsymbol{\xi}_k^t(s)\|^2\,ds\right) \notag \\
&= \frac12\int_0^\infty\sum_{i,k=1}^N\mu_{ik}(s)\bigl(\boldsymbol{\xi}_i^t(s)-\boldsymbol{\xi}_k^t(s)\bigr)^T \bigl( \partial_t\boldsymbol{\xi}_i^t(s) - \partial_t\boldsymbol{\xi}_k^t(s) \bigr)\,ds \notag \\
&= \frac12\int_0^\infty\sum_{i,k=1}^N\mu_{ik}(s)\bigl(\boldsymbol{\xi}_i^t(s)-\boldsymbol{\xi}_k^t(s)\bigr)^T \left[ \bigl(\boldsymbol{v}_i(t)-\boldsymbol{v}_k(t)\bigr) - \partial_s\bigl(\boldsymbol{\xi}_i^t(s)-\boldsymbol{\xi}_k^t(s)\bigr) \right] ds \notag \\
&= \frac12\int_0^\infty\sum_{i,k=1}^N\mu_{ik}(s)\bigl(\boldsymbol{v}_i(t)-\boldsymbol{v}_k(t)\bigr)^T\bigl(\boldsymbol{\xi}_i^t(s)-\boldsymbol{\xi}_k^t(s)\bigr)\,ds \notag \\
&\quad - \frac14\int_0^\infty\sum_{i,k=1}^N \mu_{ik}(s)\,\partial_s\|\boldsymbol{\xi}_i^t(s)-\boldsymbol{\xi}_k^t(s)\|^2\,ds.
\end{align}
}
{Summing \eqref{eq:Ev_dot_state} and \eqref{eq:Ev_dot_memory}, and applying integration by parts to the last term of \eqref{eq:Ev_dot_memory}, we arrive at:
{\small
\begin{flalign}\label{eq:dot_Ev_new}
\dot{E}_v(t) &= -\frac12\sum_{i,k=1}^N\psi(\|\boldsymbol{x}_i(t)-\boldsymbol{x}_k(t)\|)\|\boldsymbol{v}_i(t)-\boldsymbol{v}_k(t)\|^2 \notag \\
&\quad - \frac{1}{4}\sum_{i,k=1}^N \left[ \mu_{ik}(s) \|\boldsymbol{\xi}_i^t(s) - \boldsymbol{\xi}_k^t(s)\|^2 \right]_0^\infty + \frac14\int_0^\infty\sum_{i,k=1}^N \mu_{ik}'(s)\|\boldsymbol{\xi}_i^t(s)-\boldsymbol{\xi}_k^t(s)\|^2\,ds \notag \\
&= -\frac12\sum_{i,k=1}^N\psi(\|\boldsymbol{x}_i(t)-\boldsymbol{x}_k(t)\|)\|\boldsymbol{v}_i(t)-\boldsymbol{v}_k(t)\|^2 + \frac14\int_0^\infty\sum_{i,k=1}^N \mu_{ik}'(s)\|\boldsymbol{\xi}_i^t(s)-\boldsymbol{\xi}_k^t(s)\|^2\,ds\nonumber\\
&=-\frac12\sum_{i,k=1}^N\psi(\|\boldsymbol{x}_i(t)-\boldsymbol{x}_k(t)\|)\|\boldsymbol{v}_i(t)-\boldsymbol{v}_k(t)\|^2 - \frac14 Q_v(t)
\end{flalign}
}
}
{\noindent\textit{Step 2: Time-derivative of the multiplier.} Differentiating $\Phi_v(t)$ and utilizing the transport relation $\partial_t \boldsymbol{\xi}_i^t(s) = \boldsymbol{v}_i(t) - \partial_s \boldsymbol{\xi}_i^t(s)$ gives:
\begin{align*}
\dot{\Phi}_v(t) &= -\frac12\int_0^\infty\sum_{i,k=1}^N\mu_{ik}(s)\bigl(\dot{\boldsymbol{v}}_i(t)-\dot{\boldsymbol{v}}_k(t)\bigr)^T\bigl(\boldsymbol{\xi}_i^t(s)-\boldsymbol{\xi}_k^t(s)\bigr)\,ds \notag \\
&\quad -\frac12\int_0^\infty\sum_{i,k=1}^N\mu_{ik}(s)\bigl(\boldsymbol{v}_i(t)-\boldsymbol{v}_k(t)\bigr)^T\Bigl[\bigl(\boldsymbol{v}_i(t)-\boldsymbol{v}_k(t)\bigr)-\partial_s\bigl(\boldsymbol{\xi}_i^t(s)-\boldsymbol{\xi}_k^t(s)\bigr)\Bigr]\,ds.
\end{align*}
Exchanging summation indices $i \leftrightarrow k$ and substituting the velocity equation $\dot{\boldsymbol{v}}_i(t) = -[A(t)\boldsymbol{v}]_i - \boldsymbol{m}_i(t)$, where $[A(t)\boldsymbol{v}]_i := \sum\limits_{k=1}^N \psi(\|\boldsymbol{x}_i - \boldsymbol{x}_k\|)(\boldsymbol{v}_i - \boldsymbol{v}_k)$, the first term expands to:
\begin{align*}
-\frac12\int_0^\infty\sum_{i,k=1}^N\mu_{ik}(s)\bigl(&\dot{\boldsymbol{v}}_i(t)-\dot{\boldsymbol{v}}_k(t)\bigr)^T\bigl(\boldsymbol{\xi}_i^t(s)
-\boldsymbol{\xi}_k^t(s)\bigr)\,ds = -\sum_{i=1}^N \dot{\boldsymbol{v}}_i(t)^T \boldsymbol{m}_i(t) \\
&= \sum_{i=1}^N \|\boldsymbol{m}_i(t)\|^2 + \sum_{i=1}^N [A(t)\boldsymbol{v}]_i^T \boldsymbol{m}_i(t).
\end{align*}
For the second term, performing integration by parts in $s$ yields:
\begin{align*}
&-\frac12\sum_{i,k=1}^N\left(\int_0^\infty \mu_{ik}(s)\,ds\right)\|\boldsymbol{v}_i(t)-\boldsymbol{v}_k(t)\|^2 + \frac12\int_0^\infty\sum_{i,k=1}^N \mu_{ik}(s)(\boldsymbol{v}_i-\boldsymbol{v}_k)^T\partial_s(\boldsymbol{\xi}_i^t-\boldsymbol{\xi}_k^t)\,ds \\
&= -\frac12\sum_{i,k=1}^N \lambda_{ik}\|\boldsymbol{v}_i(t)-\boldsymbol{v}_k(t)\|^2 - \frac12\int_0^\infty\sum_{i,k=1}^N \mu'_{ik}(s)(\boldsymbol{v}_i(t)-\boldsymbol{v}_k(t))^T(\boldsymbol{\xi}_i^t(s)-\boldsymbol{\xi}_k^t(s))\,ds,
\end{align*}
where $\lambda_{ik} = \int_0^\infty \mu_{ik}(s)\,ds = g_{ik}(0) \ge 0$. Combining these evaluations yields:
\begin{align}\label{eq:dot_Phiv_expanded}
\dot{\Phi}_v(t) &= \sum_{i=1}^N \|\boldsymbol{m}_i(t)\|^2 + \langle A(t)\boldsymbol{v}, \boldsymbol{m}\rangle - \frac12\sum_{i,k=1}^N \lambda_{ik}\|\boldsymbol{v}_i(t)-\boldsymbol{v}_k(t)\|^2 \notag \\
&\quad - \frac12\sum_{i,k=1}^N \int_0^\infty \mu'_{ik}(s)(\boldsymbol{v}_i(t)-\boldsymbol{v}_k(t))^T(\boldsymbol{\xi}_i^t(s)-\boldsymbol{\xi}_k^t(s))\,ds.
\end{align}
}
{\noindent\textit{Step 3: Absorption of cross-terms and rigorous bounds.}
By the symmetry of communication weights, $A(t)$ is a symmetric, positive semi-definite linear operator on $(\mathbb{R}^n)^N$ satisfying, for all $\boldsymbol{u} \in (\mathbb{R}^n)^N$,
\[
\langle A(t)\boldsymbol{u}, \boldsymbol{u}\rangle = \frac{1}{2}\sum_{i,k=1}^N \psi(\|\boldsymbol{x}_i - \boldsymbol{x}_k\|)\|\boldsymbol{u}_i - \boldsymbol{u}_k\|^2 \ge 0,
\]
and the upper bound
\[
\langle A(t)\boldsymbol{u}, \boldsymbol{u}\rangle \le \frac{\psi(0)}{2}\sum_{i,k=1}^N \|\boldsymbol{u}_i - \boldsymbol{u}_k\|^2 \le N\psi(0)\sum_{i=1}^N \|\boldsymbol{u}_i\|^2 = N\psi(0)\|\boldsymbol{u}\|^2.
\]
Therefore, applying the Cauchy--Schwarz and Young inequalities for the positive semi-definite form $\langle A(t)\cdot, \cdot\rangle$, we obtain for any $\varepsilon > 0$:
\begin{align}\label{eq:Young_Laplacian}
\varepsilon \langle A(t)\boldsymbol{v}, \boldsymbol{m}\rangle &\le \frac{1}{2}\langle A(t)\boldsymbol{v}, \boldsymbol{v}\rangle + \frac{\varepsilon^2}{2}\langle A(t)\boldsymbol{m}, \boldsymbol{m}\rangle \notag \\
&\le \frac{1}{4}\sum_{i,k=1}^N \psi(\|\boldsymbol{x}_i-\boldsymbol{x}_k\|)\|\boldsymbol{v}_i-\boldsymbol{v}_k\|^2 + \frac{\varepsilon^2 N\psi(0)}{2}\sum_{i=1}^N \|\boldsymbol{m}_i(t)\|^2.
\end{align}
Notice that the first term in \eqref{eq:Young_Laplacian} is strictly absorbed by the negative instantaneous dissipation in \eqref{eq:dot_Ev_new}.
}
{Next, we establish a bound on $\sum\limits_{i=1}^N \|\boldsymbol{m}_i(t)\|^2$. Factoring the integrand of $\boldsymbol{m}_i(t)$ as
\[
\mu_{ik}(s)\bigl(\boldsymbol{\xi}_i^t(s)-\boldsymbol{\xi}_k^t(s)\bigr) = \frac{\mu_{ik}(s)}{\sqrt{-\mu'_{ik}(s)}} \cdot \sqrt{-\mu'_{ik}(s)}\bigl(\boldsymbol{\xi}_i^t(s)-\boldsymbol{\xi}_k^t(s)\bigr),
\]
and applying the Cauchy--Schwarz inequality, first in $s$ then in $k,$ we obtain:
\begin{align*}
\|\boldsymbol{m}_i(t)\|^2 &= \left\| \sum_{k=1}^N \int_0^\infty \frac{\mu_{ik}(s)}{\sqrt{-\mu'_{ik}(s)}} \cdot \sqrt{-\mu'_{ik}(s)}\bigl(\boldsymbol{\xi}_i^t(s)-\boldsymbol{\xi}_k^t(s)\bigr)\,ds \right\|^2 \\
&\le \left(\sum_{k=1}^N \int_0^\infty \frac{\mu_{ik}(s)^2}{-\mu'_{ik}(s)}\,ds\right) \left(\sum_{k=1}^N \int_0^\infty (-\mu'_{ik}(s))\|\boldsymbol{\xi}_i^t(s)-\boldsymbol{\xi}_k^t(s)\|^2\,ds\right).
\end{align*}
By Assumption~\ref{assum:g}(ii), $-\mu'_{ik}(s) \ge \delta \mu_{ik}(s) > 0$, hence $\frac{\mu_{ik}(s)^2}{-\mu'_{ik}(s)} \le \frac{1}{\delta}\mu_{ik}(s)$. Thus,
\[
\sum_{k=1}^N \int_0^\infty \frac{\mu_{ik}(s)^2}{-\mu'_{ik}(s)}\,ds \le \frac{1}{\delta}\sum_{k=1}^N \int_0^\infty \mu_{ik}(s)\,ds = \frac{1}{\delta}\sum_{k=1}^N \lambda_{ik} \le \frac{N\lambda_{\max}}{\delta} =: K^v < \infty.
\]
Summing over $i \in [N]$ directly gives:
\begin{equation}\label{eq:m_CS_bound}
\sum_{i=1}^N \|\boldsymbol{m}_i(t)\|^2 \le K^v \sum_{i,k=1}^N \int_0^\infty (-\mu'_{ik}(s))\|\boldsymbol{\xi}_i^t(s)-\boldsymbol{\xi}_k^t(s)\|^2\,ds = K^v Q_v(t).
\end{equation}
}
{For the cross-term in \eqref{eq:dot_Phiv_expanded}, since $-\mu'_{ik}(s) \ge 0$, applying Young's inequality with parameter $\alpha > 0$ yields:
\begin{flalign*}
&- \frac12\sum_{i,k=1}^N \int_0^\infty \mu'_{ik}(s)(\boldsymbol{v}_i(t)-\boldsymbol{v}_k(t))^T(\boldsymbol{\xi}_i^t(s)-\boldsymbol{\xi}_k^t(s))\,ds\\
&\le \frac{1}{4\alpha}\sum_{i,k=1}^N \left(\int_0^\infty (-\mu'_{ik}(s))\,ds\right)\|\boldsymbol{v}_i(t)-\boldsymbol{v}_k(t)\|^2 + \frac{\alpha}{4}\sum_{i,k=1}^N \int_0^\infty (-\mu'_{ik}(s))\|\boldsymbol{\xi}_i^t(s)-\boldsymbol{\xi}_k^t(s)\|^2\,ds.
\end{flalign*}
Using $\int_0^\infty (-\mu'_{ik}(s))\,ds = \mu_{ik}(0) \le \mu_{\max}^0$ and the zero-mean identity $\sum\limits_{i,k=1}^N \|\boldsymbol{v}_i - \boldsymbol{v}_k\|^2 = 2N\sum\limits_{i=1}^N \|\boldsymbol{v}_i\|^2$, we conclude:
\begin{equation}\label{eq:cross_CS_bound}
- \frac12\sum_{i,k=1}^N \int_0^\infty \mu'_{ik}(s)(\boldsymbol{v}_i(t)-\boldsymbol{v}_k(t))^T(\boldsymbol{\xi}_i^t(s)-\boldsymbol{\xi}_k^t(s))\,ds \le \frac{N\mu_{\max}^0}{2\alpha}\sum_{i=1}^N \|\boldsymbol{v}_i(t)\|^2 + \frac{\alpha}{4}Q_v(t).
\end{equation}
}
{\noindent\textit{Step 4: Exponential velocity alignment.} We define the composite Lyapunov functional:
\[
\mathcal{E}_v(t) := E_v(t) + \varepsilon \Phi_v(t), \qquad \varepsilon > 0.
\]
By Assumption~\ref{assum:g}(i), $\lambda_{ik} \ge \lambda_{\min} > 0$ for $i \ne k$, giving $\sum\limits_{i,k=1}^N \lambda_{ik}\|\boldsymbol{v}_i - \boldsymbol{v}_k\|^2 \ge 2N\lambda_{\min}\sum\limits_{i=1}^N \|\boldsymbol{v}_i\|^2$. Combining \eqref{eq:dot_Ev_new}--\eqref{eq:cross_CS_bound}:
\begin{align*}
\dot{\mathcal{E}}_v(t) &\le -\frac{1}{4}\sum_{i,k=1}^N \psi(\|\boldsymbol{x}_i-\boldsymbol{x}_k\|)\|\boldsymbol{v}_i-\boldsymbol{v}_k\|^2 - \varepsilon N\left(\lambda_{\min} - \frac{\mu_{\max}^0}{2\alpha}\right)\sum_{i=1}^N \|\boldsymbol{v}_i(t)\|^2 \\
&\quad - \left(\frac14 - \varepsilon\left(K^v + \frac{\alpha}{4}\right) - \frac{\varepsilon^2 N\psi(0)K^v}{2}\right) Q_v(t).
\end{align*}
We now fix $\alpha := \frac{\mu_{\max}^0}{\lambda_{\min}} > 0$, so that $\lambda_{\min} - \frac{\mu_{\max}^0}{2\alpha} = \frac{\lambda_{\min}}{2} > 0$. Then, choosing $\varepsilon > 0$ sufficiently small ensures that:
\[
c_Q := \frac14 - \varepsilon\left(K^v + \frac{\alpha}{4}\right) - \frac{\varepsilon^2 N\psi(0)K^v}{2} > 0,
\]
and simultaneously guarantees the  equivalence $\frac12 E_v(t) \le \mathcal{E}_v(t) \le \frac32 E_v(t)$ for all $t \ge 0$.
Since $-\mu'_{ik}(s) \ge \delta \mu_{ik}(s)$, it follows that $Q_v(t) \ge 4\delta \left(\frac14 \int_0^\infty \sum\limits_{i,k=1}^N \mu_{ik}(s)\|\boldsymbol{\xi}_i^t-\boldsymbol{\xi}_k^t\|^2\,ds\right)$, and therefore:
\begin{equation*}
\dot{\mathcal{E}}_v(t) \le -\gamma_v \mathcal{E}_v(t), \qquad \forall t \ge 0,
\end{equation*}
where $\gamma_v := \frac{2}{3}\min\left\{\varepsilon N \lambda_{\min},\, 4\delta c_Q\right\} > 0$. Gr\"onwall's inequality immediately implies:
\begin{equation}\label{eq:Vmax_final_exp}
V_{\max}(t) \le 2\max_{i\in[N]}\|\boldsymbol{v}_i(t)\| \le 2\sqrt{2E_v(t)} \le C_v e^{-\frac{\gamma_v}{2}t}, \qquad \text{where }\ C_v := 2\sqrt{6E_v(0)}.
\end{equation}
}
{\noindent\textit{Step 5: Uniform spatial boundedness.} Integrating $\dot{\boldsymbol{x}}_i(t) = \boldsymbol{v}_i(t)$ along with \eqref{eq:Vmax_final_exp} over $[0, t]$ gives:
\begin{align*}
\|\boldsymbol{x}_i(t) - \boldsymbol{x}_k(t)\| &\le \|\boldsymbol{x}_i(0) - \boldsymbol{x}_k(0)\| + \int_0^t \|\boldsymbol{v}_i(\tau) - \boldsymbol{v}_k(\tau)\|\,d\tau \\
&\le X_{\max}(0) + C_v \int_0^\infty e^{-\frac{\gamma_v}{2}\tau}\,d\tau \le X_{\max}(0) + \frac{2C_v}{\gamma_v} < \infty.
\end{align*}
Taking the supremum over $t \ge 0$ completes the proof of Theorem \ref{thm:model4}. \qed
}

\subsection{Exponential flocking for system \eqref{eq:model3}}
{For Model III ($\psi \equiv 0$), the velocity equation reads:
\[
\dot{\boldsymbol{v}}_i(t) = \sum_{k=1}^{N} \int_{0}^{\infty} g_{ik}(s)\bigl(\boldsymbol{v}_k(t-s)-\boldsymbol{v}_i(t-s)\bigr)\,ds, \qquad i \in [N],
\]
which is isomorphic to the first-order opinion model \eqref{eq:model1} under the mapping $\boldsymbol{x}_i \mapsto \boldsymbol{v}_i$.

\noindent {\bf Proof of Theorem \ref{thm:model3}.}
Applying Theorem \ref{thm:model1} directly to the velocity ensemble, there exist constants $C_v > 0$ and $\gamma_v > 0$ such that
\[
V_{\max}(t) \le C_v e^{-\gamma_v t}, \qquad \forall t \ge 0.
\]
Integrating the velocity diameter yields:
\[
X_{\max}(t) \le X_{\max}(0) + \int_0^t V_{\max}(\tau)\,d\tau \le X_{\max}(0) + \frac{C_v}{\gamma_v} < \infty, \qquad \forall t \ge 0,
\]
which immediately establishes unconditional global exponential flocking and concludes the proof. \qed
}

\section{Conclusion}\label{sec:conclusion}

In this paper, we analyzed the emergent collective behaviors of multi-agent systems driven by an infinite distributed fading memory of Volterra type. By introducing exact Dafermos past-history transformations, the original integro-differential models were rigorously reformulated into dynamical systems defined on extended product Hilbert spaces. For the first-order opinion dynamics, we proved that fading memory intrinsically induces a hidden dissipative mechanism. By constructing auxiliary multiplier functionals to extract this dissipation, we established that global exponential consensus occurs unconditionally, even in the complete absence of instantaneous communication forces. {For the second-order velocity alignment model, we established unconditional global exponential flocking for both the pure fading memory system and the mixed interaction system. In the pure memory regime, exponential velocity alignment and spatial boundedness emerge intrinsically from memory dissipation alone. When instantaneous communication is also present, we developed a unified energy-multiplier framework showing that instantaneous communication strictly assists alignment without imposing any capacity threshold on initial configurations, completely independent of the spatial dispersion or the communication tail.
}

There are several remaining issues that we did not investigate in this work. For example, our current analysis fundamentally relies on symmetric interaction topologies and uniform fading memory kernels among all agents. In reality, the interaction topology is often given by a general digraph, and different agents may possess heterogeneous cognitive memory rates or be subjected to communication noise. Thus, it would be interesting and challenging to extend our results to scenarios with asymmetric couplings, directed network topologies, or heterogeneous time-delays. Rigorous analysis for these remaining issues will be addressed in future research.

\section*{Acknowledgments}
C. Pignotti is member of Gruppo Nazionale per l'Analisi Matematica,
	la Probabilità e le loro Applicazioni (GNAMPA) of the Istituto Nazionale di
	Alta Matematica (INdAM). The work of Y.-Q. Wang was supported by the State Scholarship Fund from the China Scholarship Council in 2025 (File No. 202506060082).


\begin{thebibliography}{99}



\bibitem{Blondel2009}
V.~D.~Blondel, J.~M.~Hendrickx, and J.~N.~Tsitsiklis, ``On Krause's multi-agent consensus model with state-dependent connectivity'', \textit{IEEE Trans. Automat. Control}, 54 (2009), 2586--2597.

\bibitem{Boschi2021}
G.~Boschi, C.~Cammarota, and R.~K\"uhn, ``Opinion dynamics with emergent collective memory: the impact of a long and heterogeneous news history'', \textit{Phys. A}, 569 (2021), Paper No. 125799, 19 pp.




\bibitem{dini-2}
F. Bullo,
\newblock \emph{Lectures on Network Systems},
\newblock 1.6$^{nd}$ edition, Kindle Direct Publishing, 2022.

\bibitem{Bullo2009}
F.~Bullo, J.~Cort\'es, and S.~Mart\'inez, \textit{Distributed Control of Robotic Networks: A Mathematical Approach to Motion Coordination Algorithms}, Princeton University Press, Princeton, NJ, 2009.


\bibitem{Camazine2001}
S.~Camazine, J.~L.~Deneubourg, N.~R.~Franks, J.~Sneyd, G.~Theraulaz, and E.~Bonabeau, \textit{Self-Organization in Biological Systems}, Princeton University Press, Princeton, NJ, 2001.

\bibitem{Canuto2012}
C.~Canuto, F.~Fagnani, and P.~Tilli, ``An Eulerian approach to the analysis of Krause's consensus models'', \textit{SIAM J. Control Optim.}, 50 (2012), 243--265.

\bibitem{Carrillo2010}
J.~A.~Carrillo, M.~Fornasier, G.~Toscani, and F.~Vecil, ``Particle, kinetic, and hydrodynamic models of swarming'', in \textit{Mathematical Modeling of Collective Behavior in Socio-Economic and Life Sciences}, Birkh\"auser, Boston, 2010, pp. 297--336.

\bibitem{Castellano2009}
C.~Castellano, S.~Fortunato, and V.~Loreto, ``Statistical physics of social dynamics'', \textit{Rev. Mod. Phys.}, 81 (2009), 591--646.

\bibitem{Chepyzhov2006}
V.~V.~Chepyzhov, E.~Mainini, and V.~Pata, ``Stability of abstract linear semigroups arising from heat conduction with memory'', \textit{Asymptot. Anal.}, 50 (2006), 269--291.

\bibitem{Cho2025}
H.~Cho, S.-Y.~Ha, and M.~Kang, ``Emergent behaviors of a Kuramoto ensemble under fading memory'', \textit{J. Nonlinear Sci.}, 35 (2025), Art. 9.

\bibitem{Chiara}
Y.-P.~Choi, C.~Cicolani, and C.~Pignotti,
``Time-delayed opinion dynamics with leader-follower interactions: consensus, stability, and mean-field limits'',
Commun. Math. Sci. 24 (2026), no. 4, 1073--1102.

\bibitem{Choi2021}
Y.-P.~Choi, A.~Paolucci, and C.~Pignotti, ``Consensus of the Hegselmann--Krause opinion formation model with time delay'', \textit{Math. Methods Appl. Sci.}, 44 (2021), 4560--4579.

\bibitem{Conti2020}
M.~Conti, F.~Dell'Oro, and V.~Pata, ``Exponential decay of a first order linear Volterra equation'', \textit{Math. Eng.}, 2 (2020), 459--471.

\bibitem{Continelli}
E.~Continelli,
``Asymptotic flocking for the Cucker-Smale model with time variable time delays'',
\textit{Acta Appl. Math.}, 188 (2023), Paper No. 15, 23 pp.

\bibitem{Continelli2023}
E.~Continelli and C.~Pignotti, ``Consensus for Hegselmann--Krause type models with time variable time delays'', \textit{Math. Methods Appl. Sci.}, 46 (2023), 18916--18934.

\bibitem{Cucker2007}
F.~Cucker and S.~Smale, ``Emergent behavior in flocks'', \textit{IEEE Trans. Automat. Control}, 52 (2007), 852--862.

\bibitem{Dafermos1970}
C.~M.~Dafermos, ``Asymptotic stability in viscoelasticity'', \textit{Arch. Ration. Mech. Anal.}, 37 (1970), 297--308.

\bibitem{Dong}
J.-G.~Dong, S.-Y.~Ha, J.~Jung, Jinwook, and D.~Kim,
 ``On the stochastic flocking of the Cucker-Smale flock with randomly switching topologies'',
\textit{SIAM J. Control Optim.} 58 (2020), 2332--2353.

\bibitem{Du2025}
L.~Du, S.-Y.~Ha, and H.~Yu, ``Interplay of heterogeneous time-delays and feedback control in the Hegselmann--Krause model'', \textit{J. Differential Equations}, 423 (2025), 597--630.

\bibitem{Hask_memory}
R. Erban and J.~Haskovec, ``Impact of memory on clustering in spontaneous particle aggregation'', \textit{Siam J. Life Sciences}, 1 (2026), 262--287.

\bibitem{Giorgi2001}
C.~Giorgi, J.~E.~Mu\~{n}oz Rivera, and V.~Pata, ``Global attractors for a semilinear hyperbolic equation in viscoelasticity'', \textit{J. Math. Anal. Appl.}, 260 (2001), 83--99.


\bibitem{Ha2009}
S.-Y.~Ha and J.-G.~Liu, ``A simple proof of the  Cucker-Smale flocking
dynamics and mean-field limit'', \textit{Commun. Math. Sci.}, 2 (2009), 297--325.


\bibitem{Ha2}
S.-Y. Ha and E. Tadmor,
``From particle to kinetic and hydrodynamic descriptions of
	flocking'',
\textit{Kinet. Relat. Models}, 1, (2008), 415--435.



\bibitem{Haskovec2021}
J.~Haskovec, ``A simple proof of asymptotic consensus in the Hegselmann--Krause and Cucker--Smale models with normalization and delay'', \textit{SIAM J. Appl. Dyn. Syst.}, 20 (2021), 130--148.

\bibitem{Haskovec2020}
J.~Haskovec and I.~Markou, ``Exponential asymptotic flocking in the Cucker--Smale model with distributed reaction delays'', \textit{Math. Biosci. Eng.}, 17 (2020), 5651--5671.

\bibitem{Hegselmann2002}
R.~Hegselmann and U.~Krause, ``Opinion dynamics and bounded confidence: models, analysis and simulation'', \textit{J. Artif. Soc. Soc. Simul.}, 5 (2002), no. 3, Art. 2.


\bibitem{Jabin2014}
P.-E.~Jabin and S.~Motsch, ``Clustering and asymptotic behavior in opinion formation'', \textit{J. Differential Equations}, 257 (2014), 4165--4187.

\bibitem{Jackson2008}
M.~O.~Jackson, \textit{Social and Economic Networks}, Princeton University Press, Princeton, NJ, 2008.

\bibitem{Jadbabaie2003}
A.~Jadbabaie, J.~Lin, and A.~S.~Morse, ``Coordination of groups of mobile autonomous agents using nearest neighbor rules'', \textit{IEEE Trans. Automat. Control}, 48 (2003), 988--1001.

\bibitem{Jiang2025}
M.~Jiang, W.~Su, G.~Ren, and Y.~Yu, ``Memory-driven bounded confidence opinion dynamics: A Hegselmann--Krause model based on fractional-order methods'', arXiv preprint arXiv:2506.04701, (2025).


\bibitem{Liu2023}
Q.~Liu and L.~Chai, ``The memory influence on opinion dynamics in coopetitive social networks: Analysis, application, and simulation'', \textit{IEEE Trans. Control Netw. Syst.}, 10 (2023), 1867--1878.

\bibitem{Liu2026}
Q. Liu, A. Li, and L. Chai, ``Dynamics of opinion propagation with memory and fake news'', \textit{Automatica J. IFAC}, 185 (2026), Paper No. 112745, 14 pp.


\bibitem{OlfatiSaber2007}
R.~Olfati-Saber, J.~A.~Fax, and R.~M.~Murray, ``Consensus and cooperation in networked multi-agent systems'', \textit{Proc. IEEE}, 95 (2007), 215--233.

\bibitem{Pignotti2018}
C.~Pignotti and E.~Tr\'elat, ``Convergence to consensus of the general finite-dimensional Cucker--Smale model with time-varying delays'', \textit{Commun. Math. Sci.}, 16 (2018), 2053--2076.


\bibitem{Cartabia}
M. Rodriguez Cartabia,
``Cucker-Smale model with time delay'',
\newblock{\em Discrete Contin. Dynam. Systems}, 42 (2022), no. 5, 2409--2432.

\bibitem{Su2017}
W.~Su, G.~Chen, and Y.~Hong, ``Noise leads to quasi-consensus of Hegselmann--Krause opinion dynamics'', \textit{Automatica}, 85 (2017), 448--454.

\bibitem{Vicsek1995}
T.~Vicsek, A.~Czir\'ok, E.~Ben-Jacob, I.~Cohen, and O.~Shochet, ``Novel type of phase transition in a system of self-driven particles'', \textit{Phys. Rev. Lett.}, 75 (1995), 1226--1229.

\bibitem{Volterra1930}
V. Volterra, \emph{Theory of functionals and of integral and integro-differential equations}, Dover, New York, 1959.


\bibitem{wang2025agent}
Z. Z. Wang, J. Mao, D. Fried, and G. Neubig, ``Agent Workflow Memory'', in  \textit{Forty-second International Conference on Machine Learning}, (2025), https://openreview.net/forum?id=NTAhi2JEEE.





\bibitem{Fanqin}
F.~Zeng, X.~Xue, and Y.~Zhu,
``Critical exponent for Cucker-Smale model under group-hierarchical multi-leadership'',
\textit{Appl. Math. Lett.}, 136 (2023),  No. 108452, 7 pp.

\end{thebibliography}
\end{document}